\documentclass[11pt]{article}
\usepackage[dvips]{graphicx}
\usepackage{latexsym}
\usepackage{amsfonts}
\usepackage{color}
\usepackage{amsmath, amsthm}
\usepackage{indentfirst}
\usepackage{mathrsfs}
\usepackage{amscd}
\usepackage{amsfonts}
\usepackage{epsfig}
\usepackage{color}
\usepackage{graphics}
\usepackage{dsfont}
\usepackage{bbm}
\usepackage{fullpage}
\usepackage[numbers]{natbib}
\usepackage[T1]{fontenc}
\usepackage[utf8]{inputenc}
\usepackage{authblk}
\usepackage{setspace}
\usepackage{algpseudocode}
\usepackage{algorithm}
\floatstyle{plaintop}
\restylefloat{algorithm}
\usepackage{caption}
\usepackage{rotating}
\usepackage{multirow}
\usepackage{tabularx}
\usepackage{lscape}
\usepackage{array}
\usepackage{float}
\usepackage{url}
\usepackage{tikz}
\allowdisplaybreaks[1]
\usepackage[normalem]{ulem}

\RequirePackage{tgtermes}
\RequirePackage{newtxtext}
\RequirePackage{newtxmath}
\RequirePackage{bm}
\RequirePackage{endnotes}

\newtheorem{theorem}{Theorem}[section]
\newtheorem{lemma}[theorem]{Lemma}
\newtheorem{proposition}[theorem]{Proposition}
\newtheorem{corollary}[theorem]{Corollary}
\newtheorem{assumption}[theorem]{Assumption}

\newtheorem{remark}[theorem]{Remark}

\newtheorem{definition}[theorem]{Definition}

\newcommand{\Cov}{{\rm Cov}}
\newcommand{\Go}{\Rightarrow}

\newcommand{\PP}{P}
\newcommand{\EE}{E}

\newcommand{\NN}{\mathbb{N}}
\newcommand{\RR}{\mathbb{R}}

\newcommand{\bfp}{ m}

\newcommand{\data}{\mathcal{M}}
\newcommand{\clo}{\mathcal{O}}

\title{\bf Asymptotically Optimal Appointment Scheduling in the Presence of Patient Unpunctuality}

\author[1]{Nikolai Lipscomb\thanks{Email: ndlipscomb@gmail.com.}}
\author[2]{Xin Liu\thanks{Email: xinliu@alumni.unc.edu.}}
\author[1]{Vidyadhar G. Kulkarni\thanks{Email: vkulkarn@email.unc.edu.}}
\affil[1]{Department of Statistics and Operations Research, University of North Carolina,
            Chapel Hill, NC 27599.}
\affil[2]{Amazon.com Inc.}

\begin{document}
\maketitle

\begin{abstract} We consider the optimal appointment scheduling problem that incorporates patients' unpunctual behavior, where the unpunctuality is assumed to be time dependent, but additive. Our goal is to develop an optimal scheduling method for a large patient system to maximize expected net revenue. Methods for deriving optimal appointment schedules for large-scale systems often run into computational bottlenecks due to mixed-integer programming or robust optimization formulations and computationally complex search methods. In this work, we model the system as a single-server queueing system, where patients arrive unpunctually and follow the FIFO service discipline to see the doctor (i.e., get into service). Using the heavy traffic fluid approximation, we develop a deterministic control problem, referred to as the fluid control problem (FCP), which serves as an asymptotic upper bound for the original queueing control problem (QCP). Using the optimal solution of the FCP, we establish an asymptotically optimal scheduling policy on a fluid scale. We further investigate the convergence rate of the QCP under the proposed policy. The FCP, due to the incorporation of unpunctuality, is difficult to solve analytically. We thus propose a time-discretized numerical scheme to approximately solve the FCP. The discretized FCP takes the form of a quadratic program with linear constraints. We examine the behavior of these schedules under different unpunctuality assumptions and test the performance of the schedules on real data in a simulation study. Interestingly, the optimal schedules can involve block booking of patients, even if the unpunctuality distributions are continuous.\\

\end{abstract}

\section{Introduction}\label{sec:Intro}
The study of outpatient healthcare services has drawn increasing attention in response to the growing advocacy for preventative care. As a critical component of outpatient healthcare clinics, appointment scheduling greatly impacts the perceived and actual efficiency for both patients and doctors. Inadequate scheduling could lead to prolonged idleness for doctors, whereas an excessive number of appointments might translate into extended waiting times for patients and potential overtime for doctors. From the patient's perspective, waiting times are often associated with opportunity costs including lost wages and foregone leisure time. More critically, amidst the critical events such as COVID-19 pandemic, substantial queues in healthcare clinics present tangible health risks (ref. \cite{thunstrom2020}).

Modern technologies, including online check-ins and electronic pre-treatment surveys for patients, have equipped clinics with enhanced insights into patient arrival behaviors. This has stimulated the study of unpunctuality – the tendency of patients to arrive either early or late to their appointments. In systems handling large patient volumes, the impact of unpunctuality can be profound, given that many patients are likely to deviate from the scheduled order. This is particularly evident in scenarios such as vaccination clinics dealing with over a hundred patients daily or cancer infusion and dialysis centers attending to several dozen regular patients each day. Conventional outpatient scheduling methodologies lean on multi-stage stochastic mixed-integer programs, local search heuristics, and robust optimization techniques to generate optimal schedules. However, these methodologies often face computational challenges, becoming increasingly intractable as the model parameters, especially the number of patients to be seen, rise in magnitude (ref. \cite{ahmadi-javid2017}).

In this work, we model the appointment system as a first-in-first-out (FIFO) single-server queue, where patients come to the clinic with certain unpunctuality and get treated by a single doctor. We consider a finite time horizon $[0, T]$, e.g., representing a typical day of the clinic, and assume each patient has a deterministic scheduled appointment time. We assume an additive model of unpunctuality, that is, the difference between the actual arrival time and the scheduled arrival time of a patient is a random variable and its distribution may depend on the scheduled arrival time, but independent of the behavior of other patients (see Section \ref{sec:model} for details). Our goal is to determine the number of patients and an optimal sequence of their appointment times to maximize the profit of the clinic over $[0, T]$. The profit function consists of the linear costs of patient waiting, doctor idleness and overtime, and the rewards for service completions. The direct analysis of this queueing control problem (QCP), is intractable due to its high non-linearity. In this paper, we approach the problem using the heavy traffic fluid approximation. We consider a large-scale setting where the demand volume is high (i.e., the number of patients is large) and the system capacity can scale up in response to the demand; see Assumption \ref{htc} for the heavy traffic conditions.  We develop a fluid model that captures the mean behavior of the stochastic system, and then formulate a deterministic control problem, referred to as the fluid control problem (FCP).  We show that the optimal value of the FCP serves asymptotically as a performance upper bound to the fluid-scaled QCP as the demand volume approaches infinity. We then propose a control policy for the QCP by using the optimal solution of the FCP and show that under the proposed policy, the FCP upper bound is achieved asymptotically as the demand volume approaches infinity (see Theorem \ref{th:optimal}). To strengthen the asymptotic fluid approximation result in Theorem \ref{th:optimal}, under some suitable additional conditions, we show a central limit theorem (CLT) for the arrival process (see Proposition \ref{arrival:clt}), and further establish a CLT convergence rate for the value function under the proposed policy (see Theorem \ref{ht_error}). The incorporation of unpunctuality into the FCP leads to difficulty in deriving an analytical solution in general. We discretize the FCP as a quadratic program (QP), and show that the QP solution converges to the FCP solution as the time-discretization approaches zero. Finally, we use this QP to compute appointment schedules for various types of unpunctuality distributions, and present results using the real unpunctuality data from local clinics. We compare the performance of schedules derived from our numerical method against other baseline appointment schedules, and conclude that our method provides more efficient schedules.

This paper extends the optimal appointment scheduling in \cite{armony2019} to include the patient unpunctuality. In \cite{armony2019}, the authors consider a similar problem of scheduling appoinments for a finite customer population with customer no-shows. Here the customer no-shows can be interpreted as an extreme case of our general unpunctuality. The main results in \cite{armony2019} consider both fluid and diffusion approximations and identify asymptotically optimal schedules in both scales. The presence of unpunctuality in our current work introduces non-trivial complexity in solving the FCP, whereas the FCP in \cite{armony2019} can be easily solved with a simple analytical solution (see our Proposition \ref{fcp:no:pun}). 

The main contribution of this paper is twofold: (1) We incorporate the important unpunctuality feature in the patient arrival process and establish heavy traffic fluid approximations for the queueing system. (2) We develop a suitable FCP and a QP to numerically solve the FCP. Using the solution of the FCP, we propose a simple scheduling policy determining the number of patients and their appointment times, and show that it is asymptotically optimal in fluid scale. It comes as a surprise that the quadratic program automatically yields block scheduling of patients as an optimal appointment schedule in many scenarios. This is entirely unexpected, since the unpunctuality distributions used in the computation are continuous. In particular, in the absence of unpunctuality, our problem reduces to the one studied by \cite{armony2019}.

\section{Brief Literature Review}\label{sec:litreview}

The mathematical study of outpatient appointment scheduling can be traced back to the seminal work of \cite{bailey1952} while the examination of unpunctuality in appointment systems was first examined by \cite{white1964}. Comprehensive literature reviews of outpatient appointment scheduling are provided by \cite{cayirli2003}, \cite{gupta2008}, and \cite{ahmadi-javid2017}. The area of appointment scheduling is very broad. We will focus on relevant papers that focus on fluid and diffusion limits of queueing appointment systems, as well as the unpunctuality analysis of patients in appointment systems.

Fluid or diffusion models of appointment systems are often used when the underlying model meets suitable heavy traffic (limiting) conditions on its parameters. These are often used as approximations of actual processes that may be intractable if viewed through a discrete-event lens. Loosely speaking, the fluid model captures the mean behavior of the actual process, while the diffusion model characterizes the fluctuation of the actual process at the fluid model. An understandable and concise introduction to fluid and diffusion limits of queues can be found in Chapter 8 of \cite{gautam2012}. %\sout{Whitt (2006), examines deterministic fluid models of multiserver queues with customer abandonment.}
In \cite{lee2009}, the authors examine regularly-visiting patients and address the idleness induced by frequent hospitalization via the consideration of overbooking, where local diffusion approximations are used to derive closed-form expressions of the optimal capacity and overbooking level. \cite{chen2015} examines the fluid process of patient arrivals in complex network-like health care settings. \cite{lu2017} studies a fluid-approximation queueing network on a large hospital system for the purpose of optimally allocating cashiers and pharmacists. In \cite{zacharias2017}, the authors consider the joint problem of determining panel size and the number of appointments per day with a diffusion approximation for evaluating performance. Recently, \cite{mehrizi2022} develops the diffusion limit of multi-class advance patient scheduling and formulate the optimal scheduling problem as a Brownian control problem.

The effects of patient unpunctuality on appointment systems has been examined as early as 1964. However, the focus of incorporating unpunctuality in mathematical optimization models for appointment scheduling is a much more recent trend, likely due to the increasing availability of data provided by electronic check-ins and pre-visit surveys. \cite{tai2012} examines unpunctuality data and proposes a unique mixture distribution as a better representation of patient unpunctuality behaviour, and then develops a sequential scheduling method for developing optimal schedules. \cite{cheong2013} examines unpunctuality in a cancer infusion center and finds a mixture-exponential (Laplace) distribution as a better fit to the data than Normal-based distribution fitting. In our own work, we observe that Normal-based unpunctuality scheduling will differ greatly from Laplace-based unpunctuality scheduling. \cite{klassen2014} uses a simulation optimization framework that incorporates patient unpunctuality, and finds that a combination of block scheduling and variable-length intervals between appointments better handles unpunctuality than typical dome-shaped inter-appointment structures. \cite{schwarz2016} examines the performance of queues with time-dependent characteristics, such as arrival rates or number of servers.  \cite{samonrani2016} studies the sequencing of patients from a service-discipline perspective and address the wait-preempt dilemma induced by unpunctual patients in the system. \cite{zhu2018} considers patient unpunctuality in a system with multiple patient classes; using an optimal 2-patient model, they develop a heuristic policy in a simulation framework for improving appointment schedules. \cite{luo2016} studies the effects of patient unpunctuality and doctor consultation times on a large hospital using a data-driven simulation study. \cite{deceuninck2018} first examines outpatient scheduling problems with heterogeneous unpunctual patients and produces high-performance schedules that incorporate different patient classification. Later on \cite{deceuninck2019} is further improved by incorporating a variance-reduction technique, leveraging control variates to provide computational improvements on traditional Monte Carlo-based methods. \cite{pan2019} examines a multi-server appointment system with unpunctual patients and no-shows and formulate a two-stage stochastic programming model. Optimal appointment schedules are then derived via Bender's decomposition with a sample-average approximation. \cite{jiang2019} provides a stochastic programming model for appointment scheduling with homogeneous patient unpunctuality; the method uses Bender's decomposition of a sample-average-approximation-based mathematical program to determine optimal schedules. \cite{cayirli2019} examines the effects of environmental factors such as service time variability, no-show probability, walk-in rate on appointment systems with additional experiments focused on both patient and physician unpunctuality. \cite{mandelbaum2020} studies appointment scheduling and sequencing in time-dependent multi-server queues with patient unpunctuality. They develop a data-driven infinite-server-based heuristic and robust optimization algorithm for deriving schedules for various scenarios. \cite{pan2021-1, pan2021-2} consider unpunctual patients in multi-server systems with no-shows. They develop a two-stage stochastic mixed-integer program for deriving schedules. They extend this model by also examining the service discipline from a sequencing perspective and establish a novel index-based policy for selecting the next patient in queue. \cite{moradi2022} examines the difficulties introduced by patient unpunctuality to radiotherapy centers, developing a mixed integer program for determining the optimal sequence of patients for treatment. \cite{lipscomb2023} examines the effects of patient heterogeneity in terms of both unpunctuality and service time and how such complications can impact optimal appointment schedules. \cite{chen2023} simulates a hospital ultrasound department and examines the effects of varying patient unpunctuality and service discipline on system efficiency.

Among the above papers, \cite{armony2019} and  \cite{honnappa2015} are the most relevant to our work. The fluid model developed in this paper reduces to the fluid model studies in these papers if there is no unpunctuality. They show that in the asymptotic case it is optimal to schedule patients uniformly during the clinic hours, with a possible block schedule at the end of the clinic hours. Our results show that block scheduling of patients can be optimal throughout the day for certain unpunctuality distributions. The block scheduling phenomena is also reported by \cite{klassen2014} using entirely different tools of analysis.

\section{Queueing Model}\label{sec:model}

We consider a clinic that operates over $[0,T]$ and schedules $\bfp$ patients to be seen over this time horizon. Let $a_{i}$ be the deterministic scheduled appointment time of the $i$-th patient. Without loss of generality, assume
\[ 0 \le a_{1} \le a_{2} \le \cdots \le a_{\bfp} \le T.\]
Due to patient unpunctuality the $i$th scheduled patient actually arrives at time
\begin{align}
 T_{i} = a_{i} + U_i, \;\;\; 1 \le i \le \bfp,
 \end{align}
where $U_i$ is a random variable representing the unpunctuality of the $i$th patient. If $U_i < 0$, the patient arrives early for the scheduled appointment, if $U_i > 0$, the patient is late, and if $U_i = 0$, the patient is on time. The unpunctuality values $\{U_i\}_{i = 1}^{\bfp}$ are assumed to be independent random variables, and $U_i$ has a distribution depending on $a_i$ given by $F(t,a_i) = P(U_i \le t)$, for $t\in\RR.$
Thus, the unpunctuality of patient $i$ can depend on its scheduled appointment time $a_i$.
%Moved to assumptions section for fluid limits.
%We assume that the function $F(t, a)$ is piecewise continuous in $a$ and for each piece it is continuous uniformly for $t\in \RR$. More precisely, there exists a finite (deterministic) partition $0=\tau_0 < \tau_1< \cdots < \tau_K < \tau_{K+1} = T$ such that $\sup_{t\in\RR}|F(t, a_1) - F(t, a_2)| \to 0$ as $|a_1 - a_2|\to 0$ when $a_1$ and $a_2$ are within any partition interval $[\tau_i, \tau_{i+1}), i = 0, \ldots, k-1$ or $[\tau_K, \tau_{K+1}]$.
Let $T_{(1)} \le T_{(2)} \le \cdots \le T_{(\bfp)} $ be the order statistics of the $\bfp$ actual arrival times. We observe that $T_{(i)}$ is the time of the $i$th arrival to the clinic.

We note that the arrival times could be negative, i.e., $T_i < 0$ with nonnegative probability. This happens when patients arrive before the clinic opens at time $t=0$. Let $E(t)$ be the cumulative number of arrivals over $(-\infty,t]$, which can be formulated as
\[ E(t) = \sum_{i=1}^{\bfp} 1_{\{T_i \le t\}}, \;\; t\in\RR.\]
All patients are served by a single doctor in the order of their arrivals (i.e., FIFO service discipline). Let $\nu_i$ be the service time of patient $i$. We assume that $\{\nu_i, i\in \NN\}$ is an i.i.d. sequence with common mean $1/\mu$ and standard deviation $ \sigma$. Let $\{S(t), t \ge 0\}$ be the renewal process generated by $\{\nu_i,i \in\NN\}$, i.e.,
\[
S(t) = \sum_{i=1}^\infty 1_{\{\nu_i \le t\}}, \ \ t\ge 0.
\]
% \[
% S(t) = \max\left\{l \in\mathbb{N}: \sum_{i=1}^l \nu_i \le t\right\}, \ \ t\ge 0.
% \]
We note that  $S(t)$ is the total number of departures over $[0,t]$ if the doctor is always busy over $[0,t]$.  For $t \le 0$, we define $S(t) = 0$. We will also denote $\mu T$ as the \emph{capacity} of the system; however, depending on a choice of $m$, a system can be booked over or under capacity.

Let $Q(t)$ be the number of customers in the system at time $t \in \RR.$  We assume that the patients who arrive after time $T$ will not be allowed to enter the system. Then we have that for $t\in \RR,$
\begin{align}
Q(t) = \begin{cases} E(t\wedge T) - S(B(t)), & t\ge 0,\\
E(t), & t< 0,
\end{cases}
\end{align}
where for $t\ge 0,$
\[ B(t) = \int_0^t 1_{\{Q(s) > 0\}} ds\]
representing the cumulative busy time of the server over $[0,t]$. The idle time of the server over $[0, t]$ can then be formulated as
\[ I(t) =  t - B(t), \;\; t\ge 0.\]
%Define $B(t)=I(t)=0$ for $t<0$, extending both processes to the time before the start time $0$. 
%Finally, let $V(t)$ denote the queueing time of a customer (time spent in the system before starting service) who happens to arrive at time $t$. The process $\{V(t); t\ge 0\}$ is referred to as the virtual queueing time process. Under the FIFO service discipline, we have
%\begin{align}
%V(t) = \begin{cases}
%\sum_{j=1}^{E(t)} \nu_j-B(t) & 0 \le t\le T, \\
%\sum_{j=1}^{E(t)} \nu_j - t & t<0.
%\end{cases}
%\end{align}
%
Finally, let $O(T)$ be a nonnegative random variable representing the amount of overtime it takes to empty the system, i.e.,
\[O(T) = \inf\{t\ge T: Q(t) = 0\} - T.\]

Our goal is to determine the number of patients $\bfp$ and an optimal sequence of appointment times $\{a_i \}_{i=1}^\bfp$ to maximize the profit of the clinic. Let $r$ denote the reward of serving each patient, $c_w$ the waiting time cost rate of each patient, $c_i$ the idle time cost rate for the doctor, and $c_o$ the overtime cost rate. Let $\data = (\mu, r, c_w, c_i, c_o, \{F(t,a);t\ge0, a\in[0,T]\})$ be the collection of the system parameters that are assumed to be known/estimated for our optimization problem. Finally, our control problem is:
%\begin{equation}\label{scp}
%\begin{aligned}
%  \max \mathcal{J}(m, \{a_i\}_{i=1}^\bfp; \data)  = &  \mathbb{E}\left[ r E(T)\right] \\
%  &- \mathbb{E}\left[c_w \int_0^\infty Q(t) dt \right]\\
%  &- \mathbb{E} \left[c_i I(T) - c_o O(T) \right]\\
%  \text{subject to} \  \bfp \in \mathbb{Z}^+, \\
%   0 \leq a_1 \dots \leq a_{\bfp} \leq T.
%\end{aligned}
%\end{equation}
\begin{align}\label{scp}
\begin{aligned}
  \max \mathcal{J}(m, \{a_i\}_{i=1}^\bfp; \data) & =  \mathbb{E}\left[ r E(T) - c_w \int_0^\infty Q(t) dt - c_i I(T) - c_o O(T) \right]\\
  \text{subject to} \ & \bfp \in \mathbb{Z}^+, \ \mbox{and} \ 0 \leq a_1 \leq a_2 \leq \dots \leq a_{\bfp} \leq T.
\end{aligned}
\end{align}

We will refer to the control problem in \eqref{scp} as the queueing control problem (QCP). This is a highly non-linear optimization problem, and solving it to optimality is intractable. Several heuristics were developed and studied in \cite{lipscomb2023}. In this paper we establish asymptotically optimal solutions for a large-scale system (with large capacity $\mu T$), by exploring the fluid limits for the queueing processes as shown in the following section.

\section{Asymptotic Framework}\label{sec:asymp}

In this section, we consider a clinic that can treat a large number of patients. We will construct an asymptotic framework under which the clinic scales up to infinity.
More precisely, we introduce a parameter $n$ that represents the system scale and can be considered as the average number of patients the clinic can treat during a unit time interval (e.g., the average number of patients treated in a day).
We introduce a sequence of queueing systems as described in Section \ref{sec:model}, indexed by $n\in\mathbb{N}$.
For the $n$-th system, we append a superscript $n$ to all system processes, random variables, and parameters.
However, for simplicity, we assume that the unpunctuality random variables $\{U_i\}$ and its distribution $F(t, a)$ are independent of $n$. In particular, there are $\bfp^n$ number of patients arriving at $0 \le a^n_{1} \le a^n_{2} \le \cdots \le a^n_{{\bfp}^n} \le T$ in the $n$th system. The  i.i.d. service times $\{\nu^n_i\}_{i\in\mathbb{N}}$ have mean $1/\mu^n$ and standard deviation $\sigma^n$. The cost parameters are denoted as $r^n, c^n_w, c^n_i$ and $c^n_o.$ The stochastic processes $E^n(t), Q^n(t), S^n(t), B^n(t)$, and $I^n(t)$ are the arrival process, queue length process, potential departure process, busy time process, and the idle time process, respectively, in the $n$th system. The random variable $O^n(T)$ represents the time needed to empty the $n$th system after time $T.$ Finally, the objective function of the QCP in the $n$th system is given by
\begin{align}\label{scp_n}
\mathcal{J}^n({\bfp}^n, \{a^n_i\}_{i=1}^{\bfp^n}; \data^n) = \mathbb{E}\left[ r^n E^n(T) - c^n_w \int_0^\infty Q^n(t) dt - c^n_i I^n(T) - c^n_o O^n(T) \right],
\end{align}
%\begin{align}\label{scp_n}
%\mathcal{J}^n({\bfp}^n, \{a^n_i\}_{i=1}^{\bfp^n}; \data^n) = & \mathbb{E}\left[ r^n E^n(T) \right] \\
%& - \mathbb{E} \left[ c^n_w \int_0^\infty Q^n(t) dt \right] \\
%&- \mathbb{E} \left[c^n_i I^n(T) \right] \\
%& - \mathbb{E} \left[c^n_o O^n(T) \right],
%\end{align}
where $\data^n = (\mu^n, r^n, c_w^n, c^n_i, c_o^n, \{F(t,a); t\ge 0, a\in[0,T]\})$ is the collection of the parameters for the $n$th system.

To construct the asymptotic setting we assume the following heavy traffic conditions on the service capacity.
\begin{assumption}[Heavy traffic conditions]\label{htc} \hfill
\begin{itemize}
\item[\rm (i)] $\bfp^n = \clo(n).$

%There exists $\bar\bfp > 0$ such that
%\begin{align}\label{ht_1}
%\bfp^n = n\bar\bfp + o(n), \ \mbox{as $n\to\infty$.}
%\end{align}
%where $\bar p$ is a positive constant,
\item[\rm (ii)] There exist $\bar\mu, \bar\sigma > 0$ such that
\begin{align}\label{ht_2}
%\lim_{n\to\infty}\frac{\mu^n}{n} = \bar \mu, \ \ \mbox{and} \ \ \lim_{n\to\infty}\frac{\sigma^n}{n} = \bar\sigma,
\mu^n = n \bar\mu + o(n), \ \ \mbox{and} \ \sigma^n = \frac{\bar\sigma}{n} + o\left(\frac{1}{n}\right),  \ \mbox{as $n\to\infty$.}
\end{align}
\end{itemize}
\end{assumption}

\begin{remark} We provide some explanation on the heavy traffic conditions. 
\begin{itemize}
\item[\rm (i)] Since $\bfp^n$ is a control variable, we do not pose a convergence assumption on it; instead we only assume $\bfp^n = \clo(n).$
\item[\rm (ii)] Assumption \ref{htc} (ii) says the service times in the $n$th system are $\clo(n^{-1})$.  The service capacity of the $n$th system is thus becoming $\mu^n T = n \bar\mu T + o(n)$ and serves $\bfp^n =\clo(n)$ number of patients.

\end{itemize}
\end{remark}

% We note that in the $n$th system, the arrival process  is $O(n)$ [\textcolor{red}{what does this mean?}] and the service times are $O(n^{-1})$, which yields that the queue length process is $O(n)$ [\textcolor{red}{what does this mean?}]. Furthermore, at time $T$, the number of patients remaining in the system is at most $O(n)$, which gives the overtime $O^n(T)$ is $O(1)$. [\textcolor{red}{We are using $O$ for overtime process, as well as for ``big O" notation. Will this cause confusion?}] 
% \textcolor{blue}{Xin: Rewrite this part below. The purpose of this part is to introduce the fluid scaling.}

We note that in the $n$th system, the number of arrivals by time $t\in[0, T]$ is $\clo(n)$ and the number of patients who complete their visits by time $t$ is also $\clo(n)$ because the service times are $\clo(n^{-1})$, which yields that the number of patients in the system, i.e., the queue length process, at $t$ is at most $\clo(n)$. Furthermore, at time $T$, the number of patients remaining in the system is at most $\clo(n)$ because the total number of patients is $m^n = \clo(n)$. This says the overtime $O^n(T)$ is at most $\clo(1)$. 
Therefore, to derive meaningful limits for system processes, we define the following fluid scaled processes:
\begin{align*}
\bar{Q}^n(t)  = \frac{Q^n(t)}{n}, \
\bar{E}^n(t)  = \frac{E^n(t)}{n}, \
\bar{S}^n(t)  = \frac{S^n(t)}{n}.
%\bar{V}^n(t)  = \frac{V^n(t)}{n}.
\end{align*}
Note that the busy time and idle time processes $B^n(t)$ and $I^n(t)$, and the overtime $O^n(T)$ are unscaled because those quantities are $O(1)$. We introduce the fluid scaled objective function of the QCP:
%\begin{align}\label{scp_n_fluid}
% \bar{\mathcal{J}}^n({\bfp}^n, \{a^n_i\}_{i=1}^{\bfp^n}; \data^n) & = \frac{\mathcal{J}^n({\bfp}^n, \{a^n_i\}_{i=1}^{\bfp^n}; \data^n)}{n} \nonumber \\
% & = \mathbb{E}\left[ r^n \bar E^n(T) \right] \\
% &- \mathbb{E}\left[ c^n_w \int_0^\infty \bar Q^n(t) dt \\
% & - \mathbb{E}\left[\frac{c^n_i}{n} I^n(T) - \frac{c^n_o}{n} O^n(T) \right].
%\end{align}
\begin{align}\label{scp_n_fluid}
 \bar{\mathcal{J}}^n({\bfp}^n, \{a^n_i\}_{i=1}^{\bfp^n}; \data^n) & = \frac{\mathcal{J}^n({\bfp}^n, \{a^n_i\}_{i=1}^{\bfp^n}; \data^n)}{n} \nonumber \\
 & = \mathbb{E}\left[ r^n \bar E^n(T) - c^n_w \int_0^\infty \bar Q^n(t) dt - \frac{c^n_i}{n} I^n(T) - \frac{c^n_o}{n} O^n(T) \right].
\end{align}

To derive a meaningful limit for the fluid scaled objective function $\bar{\mathcal{J}}^n({\bfp}^n, \{a^n_i\}_{i=1}^{\bfp^n}; \data^n) $, we make the following assumption on the cost parameters $r^n, c^n_w, c^n_i$ and $c^n_o$.
\begin{assumption}\label{cost_para}
%\begin{itemize}
%\item[Case I.]
There exist $\bar r, \bar c_w, \bar c_i, \bar c_o > 0$ such that
\begin{align}\label{cost_par_scaling_I}
%\lim_{n\to\infty} n r^n = \bar r, \ \lim_{n\to\infty} n c_w^n = \bar c_w, \  \lim_{n\to\infty} c_i^n = \bar c_i, \ \lim_{n\to\infty} c_o^n = \bar c_o.
%r^n = \frac{\bar r}{n} + o\left(\frac{1}{n}\right), \ c^n_w = \frac{\bar c_w}{n}+ o\left(\frac{1}{n}\right), \ c^n_i = \bar c_i + o(1), \ c^n_o = \bar c_o + o(1),  \ \mbox{as $n\to\infty$.}
r^n = \bar r + o(1), \ c^n_w = \bar c_w + o(1), \ c^n_i =n \bar c_i + o(n), \ c^n_o = n\bar c_o + o(n),  \ \mbox{as $n\to\infty$.}
\end{align}
%\item[Case II.]
%%There exist $\bar r, \bar c_w, \bar c_i, \bar c_o > 0$ such that
%%\begin{align}\label{cost_par_scaling_II}
%%%\lim_{n\to\infty} n r^n = \bar r, \ \lim_{n\to\infty} n c_w^n = \bar c_w, \  \lim_{n\to\infty} c_i^n = \bar c_i, \ \lim_{n\to\infty} c_o^n = \bar c_o.
%%r^n = \bar r + o(1), \ c^n_w = \bar c_w + o(1), \ c^n_i =n \bar c_i + o(n), \ c^n_o = n\bar c_o + o(n),  \ \mbox{as $n\to\infty$.}
%%\end{align}
%%\item[Case III.]
%There exist $\bar r, \bar c_w, \bar c_i, \bar c_o > 0$ such that
%\begin{align}\label{cost_par_scaling_III}
%%\lim_{n\to\infty} n r^n = \bar r, \ \lim_{n\to\infty} n c_w^n = \bar c_w, \  \lim_{n\to\infty} c_i^n = \bar c_i, \ \lim_{n\to\infty} c_o^n = \bar c_o.
%r^n = \bar r + o(1), \ c^n_w = \bar c_w + o(1), \ c^n_i = o(n), \ c^n_o = o(n),  \ \mbox{as $n\to\infty$.}
%\end{align}
%\end{itemize}
\end{assumption}

\begin{remark}
%To explain the scaling in Assumption \ref{cost_para}, we note that in the $n$-th system, the arrival process is $O(n)$ and the service times are $O(n^{-1})$, which yields that the individual waiting time is $O(1)$. Furthermore, at time $T$, the number of patients remaining in the system is at most $O(n)$, which gives the overtime $O^n(T)$ is $O(1)$. Thus, in the stochastic control problem \eqref{scp_n}, the first term $E^n(T) = O(n)$, the second term representing the total waiting time is also $O(n)$, and both the third and the last terms are $O(1)$. To derive meaningful limits as $n\to\infty$ for all four terms, the cost parameters should satisfy
% \[
% \frac{r^n}{c^n_w} = O(1), \ \frac{r^n}{c^n_i} = O(n^{-1}), \ \frac{r^n}{c^n_o} = O(n^{-1}).
% \]

When the service time is $O(n^{-1})$, during an $O(1)$-length time interval, the clinic can treat $O(n)$ number of patients and earn $O(n)$ amount of reward (because individual reward satisfies $r^n = \bar r + o(1)  = O(1)$). Thus if the clinic is idle or overworks for $O(1)$ amount of time, it will lose $O(n)$ amount of reward. This explains why $c^n_i =n \bar c_i + o(n)$ and  $c^n_o = n\bar c_o + o(n),$ which have the order of $n$.
\end{remark}

Our goal is to establish an asymptotically optimal appointment schedule $(m^{n,*}, \{a^{n,*}_i\}_{i=1}^{m^{n,*}})$ for the QCP \eqref{scp_n} in the following sense:
For any appointment schedule sequence $\{(\bfp^n, \{a^n_i\}_{i=1}^{\bfp^n})\}_{n=1}^\infty$, we have
\begin{align}\label{optimal_11}
\limsup_{n\to\infty} \bar{\mathcal{J}}^n(\bfp^n, \{a^n_i\}_{i=1}^{\bfp^n}; \data^n) \le \lim_{n\to\infty}\bar{\mathcal{J}}^n(\bfp^{n,*}, \{a^n_i\}_{i=1}^{\bfp^{n,*}}; \data^n).
\end{align}

In the following Section \ref{se:fluid} we establish the fluid limits of system processes under an additional convergence assumption on $\bfp^n$. Section \ref{se:FCP} is devoted to defining the fluid control problem (FCP), which is the deterministic counterpart of the QCP. In Section \ref{se:AO} we construct an asymptotically optimal schedule using the optimal solution of the FCP. Finally, in Section \ref{se:applying} we explain how one uses the asymptotically optimal schedule in practice.

\subsection{Heavy Traffic Fluid Limits}\label{se:fluid}

In this section, we focus on the performance analysis of the system under the following additional assumption \eqref{ht_1} on $m^n$. We will study the control problem in the following sections and show that the proposed schedule satisfies this assumption.

Assume that there exists $\bar m > 0$ such that
\begin{align}\label{ht_1}
m^n = n \bar m + o(n), \ \mbox{as $n\to\infty$.}
\end{align}
Under Assumption \ref{htc} (ii) and \eqref{ht_1}, we now establish the fluid limits of $\bar E^n, \bar Q^n, I^n$, and $O^n(T)$. To this end, we introduce the relative frequency process for an appointment schedule $\{a_i^n\}_{i=1}^{m^n}$:
%Consider the following processes indexed by $n$:
%\[ A^n(t) = A(t),\]
%\[ S^n(t) = S(nt),\]
%\[ V^n(t) = V(A(nt)),\]
\begin{align}
 \bar A^n(t) = \begin{cases} 0, & t<0, \\
 \frac{1}{\bfp^n}\sum_{i=1}^{\bfp^n} 1_{\{a^n_{i} \le t\}}, & t\in [0, T],\\
 1, & t > T.
 \end{cases}
 \end{align}
 We assume that there exists a deterministic function $\{ A(t), t\in[0, T]\}$ such that as $n\to\infty$,
\begin{align}\label{app_cond}
\sup_{t\in[0, T]}\left| \bar A^n(t) -  A(t)\right| \to 0.
\end{align}
Let $ A(t) = 0$ for $t< 0$ and $ A(t) = 1$ for $t> T$. Then $\bar A^n(t) \to  A(t)$ uniformly for $t\in \RR$ and $\{ A(t), t\in \RR\}$ is a CDF (see the proof of Proposition \ref{prop:fluid-limits} (i)).
Finally, we introduce the following Stieltjes-convolution of $F(\cdot, \cdot)$ and $\bar\bfp A(\cdot)$:
%
%This is the limit of the empirical distribution of the appointment schedule.Let
%\[ H = F*a\]
%let Stieltjes-convolution of $F$ and $a$, that is
\begin{align}\label{fluid_arrival}
 H(t) & = \bar\bfp \int_{-\infty}^\infty F(t-s, s)d A(s) \\& = \bar\bfp \int_0^T F(t-s, s) d A(s), \;\; t\in \RR.
\end{align}

%Thus $a^n(t)$ is the number of appointments over $(-\infty,t]$.
%Then consider the following processes under fluid scaling:
%\[ \bar{A}^n(t) = A^n(t)/n,\]
%\[ \bar{S}^n(t) = S^n(t)/n = S(nt)/n,\]
% \[ \bar{V}^n(t) = V^n(t)/n = V(A(nt))/n.\]
% \[ \bar{a}^n(t) = a^n(t)/n.\]
%\textcolor{red}{{\bf Assumptions:} We make the following mathematical assumptions for our model:
%\begin{enumerate}
%  \item We assume that $0$ is in the interior of the convex hull for the support of the unpunctuality random variable for any patient. That is, $0 < F(0,a) < 1$ for any $a \in [0,T]$.
%  \item We assume that the function $F(t, a)$ is piecewise continuous in $a$ and for each piece it is continuous uniformly for $t\in \RR$. More precisely, there exists a finite (deterministic) partition $0=\tau_0 < \tau_1< \cdots < \tau_K < \tau_{K+1} = T$ such that $\sup_{t\in\RR}|F(t, a_1) - F(t, a_2)| \to 0$ as $|a_1 - a_2|\to 0$ when $a_1$ and $a_2$ are within any partition interval $[\tau_i, \tau_{i+1}), i = 0, \ldots, k-1$ or $[\tau_K, \tau_{K+1}]$.
%  \item We assume there exists a deterministic function $A(t), t\in[0, T]$ such that
%\[\sup_{t\in[0, T]}\left| \bar A^n(t) - A(t)\right| \to 0.\]
%Let $A(t) = 0$ for $t< 0$ and $A(t) = 1$ for $t> T$. Then $\bar A^n(t) \to A(t)$ uniformly for $t\in \RR$ and $\{A(t), t\in \RR\}$ is a CDF (see the proof of Lemma \ref{lm:fluid-arrival}).
%\end{enumerate}}
We pose the following assumption on the unpunctuality CDF $F(t,a)$ to ensure the convergence of $\int_0^T F(t-s, s)d\bar A^n(s)$ to $H(t)$ (see the proof of Proposition \ref{prop:fluid-limits}(i)).  
\begin{assumption}\label{unp_dist}
We assume that the function $F(t, a)$ is piecewise continuous in $a$ and for each piece it is continuous in $a$ uniformly for all $t\in \RR$. That is, there exists a finite (deterministic) partition $0=\tau_0 < \tau_1< \cdots < \tau_L  = T$, where $L\in\NN$ such that $\sup_{t\in\RR}|F(t, a_1) - F(t, a_2)| \to 0$ as $|a_1 - a_2|\to 0$ when $a_1$ and $a_2$ are within any partition interval $[\tau_i, \tau_{i+1}), i = 0, \ldots, L-2$ and $[\tau_{L-1}, T]$.
\end{assumption}

For an appointment schedule $\{a^n_i\}_{i=1}^{\bfp^n}$ satisfying \eqref{app_cond}, we establish the following limits for the fluid scaled processes. Recall that for the one-dimensional Skorokhod problem associated with $\{x(t), t\ge 0\}$, its unique solution is given by $(z(t), y(t)), t\ge 0$, where 
\begin{align}
y(t) & = \sup_{0\le s \le t} \max\{-x(s), 0\} \equiv \Psi(x)(t), \\
z(t) & = x(t) + y(t) \equiv \Phi(x)(t).
\end{align}
Here the funtional $(\Phi, \Psi)$ are known as the one-dimensional Skorokhod map. See Definition \ref{sm:def} and Theorem \ref{sm:repres} in Appendix for more information.
\begin{proposition}\label{prop:fluid-limits} We present the fluid limits for $(\bar E^n, \bar Q^n, I^n, O^n)$ as $n\to\infty$ as follows. 
\begin{itemize}
\item[\rm (i)] Under \eqref{ht_1}, \eqref{app_cond} and Assumption \ref{unp_dist}, as $n\to\infty$, almost surely,
\begin{align}\label{eq:fluid-arrival}
\sup_{t\in \RR} |\bar E^n(t) -  H(t)| \to 0. 
\end{align}
%[\textcolor{red}{is this over $t \in \RR$ or $t \in [0,T]$?}] \textcolor{blue}{Xin: It is actually for $t\in\RR.$ See (12) for $t\in\RR.$}
\item[\rm (ii)] Under Assumptions \ref{htc} (ii) and \ref{unp_dist}, \eqref{ht_1} and \eqref{app_cond},
\[
O^n(T)\to  \frac{q(T)}{\bar\mu}, \ \mbox{almost surely.}
\]

\item[\rm (iii)] Under the same condition as in (ii), as $n\to\infty$, for any $\tau >0,$ almost surely,
\begin{align*}
\sup_{t\in[0,\tau]} |(\bar Q^n(t), I^n(t)) - (q(t), i(t))| \to 0, 
\end{align*}
where $(q(t), i(t))$ is given as follows:
\begin{itemize}
    \item[(a)] for $t\in [0, T]$, $\{(q(t), \bar\mu i(t)), t\in [0, T]\}$ is the unique solution to the Skorokhod problem associated with $\{H(t)-\bar\mu t, t\in [0, T]\}$;
    \item[(b)] for $t\in \left(T, T+ {q(T)}/{\bar\mu}\right]$, $(q(t), i(t))=(q(T) - \bar\mu (t-T), i(T))$;
    \item[(c)] for $t> T+ {q(T)}/{\bar\mu} $, $(q(t), i(t))=\left(0, i(T) + t- T- {q(T)}/{\bar\mu}\right)$.
\end{itemize}
%  And for $t\ge T$,
% \[
% (q(t), i(t)) = \begin{cases}
% (q(T) - \bar\mu (t-T), i(T)), & t\in t_\alpha,\\
% \left(0, i(T) + t- T- \frac{q(T)}{\bar\mu}\right), & t > t_\beta.
% \end{cases}
% \]
% where $t_\alpha = \left[T, T+ \frac{q(T)}{\bar\mu}\right]$, and $t_\beta = T+ \frac{q(T)}{\bar\mu}$
% \begin{itemize}
% \item[\rm (a)] %For $t\le 0$, $(q(t), i(t)) = (\tilde H(t), 0).$ \item[\rm (b)] 
% For $t\in [0, T]$, $\{(q(t), \bar\mu i(t)), t\in [0, T]\}$ is the unique solution to the Skorokhod problem associated with $\{H(t)-\bar\mu t, t\in [0, T]\}$. More precisely,  $\{(q(t), \bar\mu i(t)), t\in [0, T]\}$ satisfies, for $t\in[0, T],$
% \begin{align*}
% q(t) =  H(t) - \bar\mu t + \bar\mu i(t),
% \end{align*}
% and $i(0) = 0$, $i(t)$ is nondecreasing, and it increases only when $q(t) =0.$ Equivalently, we have
% \begin{align}\label{sm}
% i(t) & = {\bar\mu}^{-1} \sup_{0\le s \le t} \max\{- H(s)+\bar\mu s, 0\}, \\ q(t) & =  H(t) - \bar\mu t + \bar\mu i(t), \ \ t\in[0, T].
% \end{align}
% \item[\rm (c)] For $t\ge T$,
% \[
% (q(t), i(t)) = \begin{cases}
% (q(T) - \bar\mu (t-T), i(T)), & t\in t_\alpha,\\
% \left(0, i(T) + t- T- \frac{q(T)}{\bar\mu}\right), & t > t_\beta.
% \end{cases}
% \]
% where $t_\alpha = \left[T, T+ \frac{q(T)}{\bar\mu}\right]$, and $t_\beta = T+ \frac{q(T)}{\bar\mu}$
% \end{itemize}

\end{itemize}
\end{proposition}

\begin{corollary}\label{opti_prop} Under Assumptions \ref{htc} (ii), \ref{cost_para} and \ref{unp_dist}, \eqref{ht_1} and \eqref{app_cond}, as $n\to\infty$,
\begin{align*}
\bar{\mathcal{J}}^n(\bfp^n, \{a^n_i\}_{i=1}^{\bfp^n}; \data^n) \to \bar r  H(T) - \bar c_w \int_0^\infty q(t) dt - \bar c_i i(T) - \bar c_o \frac{q(T)}{\bar\mu},
\end{align*}
where $ H(t)$ and $(q(t), i(t))$ are the fluid limits of the arrival process, the queue length process, and the idle process as in Proposition \ref{prop:fluid-limits}.
\end{corollary}

\begin{remark} From Proposition \ref{prop:fluid-limits} (iii), in the fluid limit, when $t<0$, the server hasn't started to work, so the idle time process has value $0$ and the queue length equals $ H(0)$ at $t=0$. When $0\le t \le T$, the fluid limit is a fluid $G_t/GI/1$ queue with initial value $ H(0)$, time-varying cumulative arrivals $ H(t)- H(0)$, and service rate $\bar\mu$.  When $t>T$, there are no more external arrivals. The queue length is now decreasing linearly with rate $\bar\mu$ till it becomes zero at time $T+ q(T)/\bar\mu$ and stays zero from then on. %Further, we note that $q(t)$ and $\bar\mu i(t)$ represent the Skorokhod map for RCLL function $ H(t) - \bar\mu t$.
\end{remark}

\begin{remark}
 The established fluid limits suggest when $n$ is large we can consider a deterministic control problem that selects $\{ A(t), t\in [0, T]\}$ to maximize the objective function $\bar r  H(T) - \bar c_w \int_0^\infty q(t) dt - \bar c_i i(T) - \bar c_o {q(T)}/{\bar\mu}.$ We will focus on this deterministic control problem in the next Subsection \ref{se:FCP} and refer to it as the fluid control problem (FCP).
\end{remark}

\subsection{Fluid Control Problem (FCP)}\label{se:FCP}

We construct the FCP that is the deterministic counterpart of the original stochastic QCP. Denote by $\bar\data = (\bar\mu, \bar r, \bar c_w, \bar c_i, \bar c_o, \{F(t,a);t\ge0, a\in[0,T]\})$ the fluid limits of the system parameters.

%\begin{definition}[FCP]\label{df:fcp}
%The FCP for the given system parameter $\bar\data$ is to choose an appointment profile $A=\{A(t), t\in [0, T]\}$ to maximize
%\begin{align}\label{fluid-obj}
% J(A; \bar\data) = & \bar r H(T) \\
% &- \bar c_w \int_0^\infty q(t) dt \\
% &- \bar c_i i(T) - \bar c_o \frac{q(T)}{\bar \mu},
%\end{align}
%subject to
%\begin{align}
%& H(t)  =  \int_0^T F(t-s, s) dA(s), \;\; t\in \RR, \label{fcp1:const1}\\
%& (q(t), i(t))  = \begin{cases}
% (H(t), 0), & t <0, \\
%(\Phi(H - \bar\mu \iota)(t), & \\
%\Psi(H - \bar\mu \iota)(t)), & t\in [0, T], \\
%(q(T) - \bar\mu (t-T), & \\ 
%i(T)), & t\in t_\alpha,\\
%(0, i(T) & \\+ (t- T- \frac{q(T)}{\bar\mu})), & t > t_\beta,
%\end{cases} \label{fcp1:const2}
%\end{align}
%with $A(0) = 0$ and $A(t)$ is RCLL nondecreasing over $[0, T]$.

%\end{definition}
\begin{definition}[FCP]\label{df:fcp}
The FCP for the given system parameter $\bar\data$ is to choose an appointment profile $A=\{A(t), t\in [0, T]\}$ to maximize
\begin{align}\label{fluid-obj}
 J(A; \bar\data) = \bar r H(T) - \bar c_w \int_0^\infty q(t) dt - \bar c_i i(T) - \bar c_o \frac{q(T)}{\bar \mu},
\end{align}
subject to
\begin{align}
& H(t)  =  \int_0^T F(t-s, s) dA(s), \;\; t\in \RR, \label{fcp1:const1}\\
& (q(t), i(t))  = \begin{cases}
 (H(t), 0), & t <0, \\
(\Phi(H - \bar\mu \iota)(t), \Psi(H - \bar\mu \iota)(t)), & t\in [0, T], \\
(q(T) - \bar\mu (t-T), i(T)), & t\in (T, T+ q(T)/\bar\mu],\\
(0, i(T) + (t- T- q(T)/\bar\mu)), & t > T+ q(T)/\bar\mu,
\end{cases} \label{fcp1:const2}\\
& \mbox{$A(0) = 0$ and $A(t)$ is RCLL nondecreasing over $[0, T]$.}\label{fcp1:const3}
\end{align}
\end{definition}

The constraints \eqref{fcp1:const1} and \eqref{fcp1:const2} are corresponding to the fluid limits in Proposition \ref{prop:fluid-limits}. 
% The functionals $\Phi$ and $\Psi$ are called the Skorokhod maps associated with $H(t) - \bar\mu t$ introduced in \eqref{sm} and are defined as
% \begin{equation}\label{fcp1:phi}
%   \Phi(H - \bar\mu \iota)(t) = H(t) - \bar\mu t + \bar\mu i(t),
% \end{equation}
% and
% \begin{equation}\label{fcp1:psi}
%   \Psi (H- \bar\mu \iota)(t) = \frac{1}{\bar\mu} \displaystyle \sup_{0 \leq s \leq t} \max \{\bar\mu s - H(s), 0  \}.
% \end{equation}
We note that the control $A(t)$ is not required to be a CDF. This relaxation combines the two control variables $\bfp^n$ and $\{a^n_i\}_{i=1}^{\bfp^n}$ of the QCP into a single control variable $\{A(t), t\in [0, T]\}.$ The terminal value $A(T)$ would represent the number of appointments, and the normalized control $A(t)/A(T)$ provides the distribution of the appointment times.
\begin{lemma}\label{FCP:finiteoptimalvalue}
The optimal value of the FCP $J^*(\bar{\mathcal{M}}) \equiv \sup_A J(A; \bar{\mathcal{M}}) < \infty.$
\end{lemma}

However, for general unpunctuality distribution $F(t; a)$, the FCP is difficult to solve to optimality with the exception of a few special cases of unpunctuality (zero unpunctuality, or certain types of uniform distributions). However, it can be solved numerically to arbitrary accuracy, if we discretize the time. In the next two subsections, we derive the exact optimal solution of the FCP for two special cases: the zero unpunctuality and uniform unpunctuality. In Section \ref{sec:QP} we solve the discretized FCP.

\subsubsection{Special Case: Zero Unpunctuality} \label{sec:NU}
We consider the special case of zero unpunctuality, that is, when all patients arrive at their scheduled appointment times. Here $H(t) = A(t)$ for all $t\in \RR$ and the FCP is simplified to choose a RCLL nondecreasing function $H=\{H(t), t\in [0, T]\}$ satisfying $H(0)=0$ to maximize
\begin{equation}
\begin{aligned}
    J(H; \bar\data) = & \bar r H(T) - \bar c_w \int_0^\infty q(t) dt - \bar c_i i(T) - \bar c_o \frac{q(T)}{\mu},
\end{aligned}
\end{equation}
subject to the constraint \eqref{fcp1:const2}. Using the properties of the Skorokhod map, this control problem can be solved explicitly (see \cite{armony2019}).

\begin{proposition}\label{fcp:no:pun}
The FCP with no unpunctuality is equivalent to the variational problem that selects a RCLL nondecreasing function $H=\{H(t); t\in[0, T] \}$ satisfying $H(0)=0$ and $H(t)\ge \bar\mu t$ for $t\in [0, T]$ to maximize
\begin{equation} \label{eq:JH}
\begin{aligned}
\check J(H; \bar\data) = &  (\bar r- \bar c_o/\bar \mu + \bar c_w T) H(T)  - \frac{\bar c_w }{2\bar \mu}H(T)^2 - \bar c_w\int_0^T H(t) dt,
\end{aligned}
\end{equation}
and it admits the following optimal solution: For $t\in[0, T]$,
\begin{align}
H^*(t) = \begin{cases}
\bar \mu t, & \mbox{if $\bar r \le \bar c_o/\bar \mu$}, \\
\bar \mu t + \frac{\bar \mu(\bar r-\bar c_o/\bar \mu)}{\bar c_w}1_{\{t=T\}}, & \mbox{if $\bar r > \bar c_o/\bar \mu$}.
\end{cases}
\end{align}
\end{proposition}

%The fluid control problem (FCP) is to choose $H = \{H(t); t\le T\}$ to maximize the objective functional $J(H).$

%{\bf Remark:} Consider the class of linear functions $H(t) = \theta t$ for $\theta \ge \mu$. Then we have
%\begin{align*}
%\tilde J(H; \mu) = f(\theta) =  (r- c_O/\mu + c_w T)T\theta - \frac{c_w T^2}{2\mu}\theta^2 - c_w T^2 \theta/2.
%\end{align*}
%This is a concave function of $\theta$ that is maximized at
%\[ \hat{\theta} = \mu \cdot \frac{r-c_O/\mu + c_W T/2}{c_W T}.\]
%Hence the optimal $\theta$ is
%\[ \theta^* = \max\{\hat{\theta}, \mu\}.\]
%Thus $\theta^* = \mu$ if $r < c_O/\mu + c_W T/2$; otherwise it equals $\hat{\theta}$. So contrary to what we thought yesterday, $\theta^*$ increases smoothly with $r$. It does not jump to $\infty$.\\
%
%Of course the optimal $H$ need not be linear. In fact, it may be optimal to use $H(t) = \mu t$ for $0 \le t < T$ and $H(T) = \mu T + x$. That is, bring in $x$ patients just before $T$. Then the net reward from this policy will be
%\[ r \mu T + rx - c_Ox/\mu- c_Wx^2/(2\mu).\]
%This is maximized at
%\[ x^* = \begin{cases}
%\frac{r\mu - c_O }{c_w}, & \mbox{if $r > c_O/\mu$}, \\
%0, & \mbox{if $r \le c_O/\mu$}.
%\end{cases}\]
%It will be good to compare the total net reward of this policy with the $\theta^*$ policy.

The corresponding optimal state process and the FCP optimal value $\check{J}^*(\bar\data) = \displaystyle \max_{H} \check J(H; \bar\data)$ are given as follows:
\begin{itemize}
\item[\rm (i)] When $\bar r \le \bar c_o/\bar \mu$, $q^*(t) = 0$ for all $t\ge 0$, and $\check{J}^*(\bar\data) = \bar r\bar \mu T.$
\item[\rm (ii)] When $\bar r > \bar c_o/\bar \mu$,
\begin{align*}
q^*(t) = \begin{cases}
0, & t\in [0, T) \\
\frac{\bar \mu(\bar r-\bar c_o/\bar \mu)}{\bar c_w} -\bar  \mu(t-T), & t \in [T, T+ \frac{(\bar r-\bar c_o/\bar \mu)}{\bar c_w}], \\
0, & t> T+ \frac{(\bar r-\bar c_o/\bar\mu)}{\bar c_w},
\end{cases}
\end{align*}
and $\check{J}^*(\bar\data)= \bar r\bar \mu T - \frac{\bar \mu[(\bar r-\bar c_o/\bar\mu)^+]^2}{2\bar c_w}$.
\end{itemize}

\begin{remark}{\rm
The optimal solution $H^*(t)$ follows the intuition that, over $[0,T)$, if the rate of patient arrivals matches the rate of patient departures, the system is in perfect equilibrium, accruing no wait costs nor idle time costs over this time horizon. The only time a wait cost and overtime cost would occur is if the reward is sufficiently high relative to the overtime and wait cost. As the wait cost of the system grows quadratically in $q(T)$, a finite amount of overbooking is guaranteed for finite $r$.}
\end{remark}

\begin{remark}
{\rm The optimal solution of the arrival function in Proposition \ref{fcp:no:pun} matches the fluid optimal schedule in \cite{armony2019}. In \cite{armony2019}, it assumes punctual arrivals but with probabilistic no-shows, and the objective function is to minimize the penalized waiting time and overtime. Our proof of Proposition \ref{fcp:no:pun} is also similar to that in \cite{armony2019}.  }
\end{remark}

%[\textcolor{red}{Should we mention that this matches with the results in [2], Harsha's paper?}]

\subsubsection{Special Case: Uniform Unpunctuality} \label{sec:UU}

Recall that if $\bar r \le \bar c_o/\bar \mu$, then $H^*(t) = \bar\mu t$ is an optimal solution of the FCP without considering unpunctuality. One possibility for finding the optimal control $A^*(t)$ analytically is to have the control satisfy $\int_0^T F(t-s,s) dA^*(s) = \bar\mu t$. The following corollary is a direct result of Proposition \ref{fcp:no:pun}. We thus omit the proof. 

\begin{corollary}\label{fcp:uni:pun}
Suppose $\bar r \le \bar c_o/\bar \mu$ and there exists RCLL nondecreasing $A(t)$ such that $\int_0^T F(t-s,s) dA(s) =\bar \mu t$ for $t \in [0,T]$. Then $A(t)$ is an optimal control to the FCP.
\end{corollary}

We consider the situation where the unpunctuality is uniformly distributed over the interval $[-a, b]$ for each patient, where $a, b > 0$, independent of their arrival time. Our goal is to construct an RCLL nondecreasing function $A$ such that $\bar\mu t = \int_0^T F(t-s) dA(s)$, where $F(t) = (t+a)/(b+a), \; t \in [-a,b],$ is the CDF of the uniform unpunctuality time. 
%This will be optimal since $H^*(t) =\bar \mu t$ will accrue a cost of 0.

Motivated from the numerical experiments, we consider piecewise constant $A(t)$ similar to the CDF of a discrete random variables. Note that if $A$ puts some positive mass $u$ at a time point $s_0$ and $[-a, b] \subset [0, T]$, in the convolution
\begin{align}
\int_0^T F(t-s) d A(s) =
\begin{cases}
0, &  \mbox{if} \ t-s_0 < -a, \\
\frac{(t-s_0+a)u}{b+a}, & \mbox{if} \ -a \le t-s_0 \le b,\\
u, & \mbox{if} \ t-s_0 > b,
\end{cases}
\end{align}
%[\textcolor{red}{should there be an integral on the LHS?}]
which says the influence of the mass $u$ at time $s_0$ is spreading uniformly over the interval $[s_0-a, s_0+b].$ Following this observation, we consider the function $A(t)$ that jumps at points $a, 2a+b, 3a+2b, \ldots, Na+(N-1)b$, where $N = \max\{n\ge 1: T- [na + (n-1)b] < a+b\}$, and the mass at each point is set to be $\bar\mu(a+b).$ Then when $T= N(a+b)$, one can check that $H^*(t) = \int_0^T F(t-s)d A(s) = \bar\mu t, \;\; 0 \le t \le T.$

%From Proposition \ref{th:br} (ii), for $t\in [0, T]$,
%\begin{align}\label{idle-1}
%i(t) = \frac{1}{\mu}(q(t) + \mu t - H(t)).
%\end{align}
%Using \eqref{fcp1:const2} and \eqref{idle-1}, the objective function \eqref{fluid-obj} becomes
%\begin{align*}
%J(A) & = r H(T) - c_w \int_0^T q(t) dt - c_w \int_T^\infty q(t)dt - \frac{c_i}{\mu}(q(T) + \mu T - H(T)) - c_o \frac{q(T)}{\mu} \\
%& = \left(r + \frac{c_i}{\mu}\right)  H(T) - c_w \int_0^T q(t) dt - \frac{c_w}{2\mu} q(T)^2 - \frac{c_i + c_o}{\mu} q(T) - c_i T.
%\end{align*}
%Consequently, the FCP is equivalent to select a function $A=\{A(t), t\in [0, T]\}$ to maximize
%\begin{align}\label{fcp2}
%\tilde J(A) = \left(r + \frac{c_i}{\mu}\right) H(T) - c_w \int_0^T q(t) dt  - \frac{c_i + c_o}{\mu} q(T) - \frac{c_w}{2\mu} q(T)^2
%\end{align}
%subject to
%\begin{align}
%& H(t)  =  \int_0^T F(t-s, s) dA(s), \;\; t\in \RR, \label{fcp2:const1}\\
%& q(t) = \Phi(H - \mu \iota)(t), \ \ t\in[0, T], \label{fcp2:const2}\\
%& \mbox{$A(0) = 0$ and $A(t)$ is RCLL nondecreasing over $[0, T]$.}\label{fcp2:const3}
%\end{align}

%\textcolor{red}{***The Skorokhod map $(\Phi, \Psi)$ is defined as follows:
%\[
%i(t) = \Psi(H - \mu \iota)(t) = \mu^{-1} \sup_{0\le s \le t} \max\{-H(s)+\mu s, 0\}, \ q(t) = \Phi(H - \mu \iota)(t) = H(t) - \mu t + \mu i(t),
%\]
%where $\iota$ is the identity map from $[0,\infty)$ to itself. For the discretized functions with $0 = t_0 < t_1 < t_2 < \cdots < t_K = T$, let
%\[
%m_k = H(t_k) - \mu t_k, \ 0\le k \le K.
%\]
%Then for $t\in [t_k, t_{k+1}),$
%\[
%i(t) =  \mu^{-1} \sup_{0\le l \le k} \max\{-m_l, 0\}, \ q(t) = m_k + \mu i(t),
%\]
%***}

\subsection{Asymptotically Optimal Schedules}\label{se:AO}

 Assume that $\{A^*(t), t\in [0, T]\}$ is an optimal solution of the FCP, and recall that $J^*(\bar\data)$ denotes the optimal value of the FCP associated with the parameters $\bar\data$. We construct an appointment schedule for the $n$th system as follows. We let $\bfp^{n,*} =  \lfloor nA^*(T) \rfloor$, and define
\begin{align}\label{policy}
a^{n,*}_k = \inf\left\{t\in [0, T]: \frac{A^*(t)}{A^*(T)} \ge  \frac{k}{\bfp^{n,*}}\right\}, \ k =1, \ldots, \bfp^{n,*}.
\end{align}
Denote by $(\bfp^{n,*}, \{a^{n,*}_i\}_{i=1}^{\bfp^{n,*}})$ the above schedule. In the following we show that $(\bfp^{n,*}, \{a^{n,*}_i\}_{i=1}^{\bfp^{n,*}})$ is asymptotically optimal.

\begin{theorem}[Asymptotic optimality]\label{th:optimal}
Under Assumption \ref{htc} (ii), \ref{cost_para} and \ref{unp_dist}, the proposed appointment schedule $\{(\bfp^{n,*}, \{a^{n,*}_i\}_{i=1}^{\bfp^{n,*}})\}_{n=1}^\infty$ satisfies
\begin{align}\label{optimal_2}
\lim_{n\to\infty} \bar{\mathcal{J}}^n(\bfp^{n,*}, \{a^{n,*}_i\}_{i=1}^{\bfp^{n,*}}; \data^n) =J^*(\bar\data).
\end{align}
Furthermore, for any appointment schedule sequence $\{(\bfp^n, \{a^n_i\}_{i=1}^{\bfp^n})\}_{n=1}^\infty$, we have
\begin{align}\label{optimal_1}
\limsup_{n\to\infty} \bar{\mathcal{J}}^n(\bfp^n, \{a^n_i\}_{i=1}^{\bfp^n}; \data^n) \le J^*(\bar\data).
\end{align}

\end{theorem}

We next investigate the convergence rate of the QCP under the asymptotically optimal appointment schedule $\{(\bfp^{n,*}, \{a^{n,*}_i\}_{i=1}^{\bfp^{n,*}})\}_{n=1}^\infty$.
We first establish the following central limit theorem (CLT) for the arrival process. Define the diffusion-scaled arrival process 
\[\hat E^n(t) = \sqrt{n}(\bar E^n(t) - H(t)), t\ge 0.\]
Assume that $F(t, a)$ is Lipschitz continuous in $t$ uniformly for all $a\in[0, T]$, i.e., there exists a $c_0>0$ such that 
\begin{align}\label{F:Lip}
|F(t, a) - F(s, a)| \le c_0 |t-s|, \ \ t, s\in [0, T].
\end{align}
\begin{proposition}[CLT for the arrival process]\label{arrival:clt}
Under \eqref{F:Lip}, and assuming that the proposed appointment schedule $\{(\bfp^{n,*}, \{a^{n,*}_i\}_{i=1}^{\bfp^{n,*}})\}_{n=1}^\infty$ satisfies 
\begin{align}\label{clt:cond:1}
\lim_{n\to\infty}\frac{1}{n}\sum_{i=1}^{m^{n,*}}F(t\wedge s -a^{n,*}_i, a^{n,*}_i)-F(t-a^{n,*}_i, a^{n,*}_i)F(s-a^{n,*}_i, a^{n,*}_i)) = \sigma^2(t, s) > 0, \ t, s \ge 0.
\end{align}
then $\hat E^n$ converges weakly to a Gaussian process with mean $0$ and covariance function $\sigma^2(t, s)$.
\end{proposition}

Strengthening Assumption \ref{cost_para}, we assume 
\begin{align}\label{cost:para:clt}
\sqrt{n}(r^n-\bar r) = O(1), \sqrt{n}(c^n_w -\bar c_w) = O(1), \sqrt{n}\left(\frac{c^n_i}{n} - \bar c_i\right) = O(1), \sqrt{n}\left(\frac{c^n_o}{n} - \bar c_o\right) = O(1).
\end{align}
\begin{theorem}\label{ht_error}
Under Assumptions \ref{htc} (ii) and \ref{cost_para}, and \eqref{F:Lip}, and if the appointment schedule $\{(\bfp^{n,*}, \{a^{n,*}_i\}_{i=1}^{\bfp^{n,*}})\}_{n=1}^\infty$ satisfies \eqref{clt:cond:1} and the system parameters satisfy \eqref{cost:para:clt}, we have
\begin{align}
{\mathcal{J}}^n(\bfp^{n,*}, \{a^{n,*}_i\}_{i=1}^{\bfp^{n,*}}; \data^n) = n J^*(\bar\data) +  O\left(\sqrt{n}\right).
\end{align}
\end{theorem}

\begin{remark}
In \cite{honnappa2015}, the authors consider the setting where there are $n$ arrivals during a time interval $[- T_0, T]$ independently following a CDF $G(t), t\in [-T_0, T]$ with $T_0 >0.$ The CLT result for the arrival process appears to be $W^0\circ G$, where $W^0$ is a standard Brownian bridge. Such result is a special case of our Proposition \ref{arrival:clt}, where $F(t-a, a) = G(t), t\in [-T_0, T]$. We note that $F(t-a_i^n, a_i^n) = P(U_i \le t-a_i^n) = P(T^n_i \le t)$, representing the CDF of the $i$-th arrival time $T^n_i$. Thus under the condition $F(t-a, a) = G(t), t\in [-T_0, T]$, the arrival times are independently sampled from the distribution $G(t), t\in [-T_0, T]$, matching the condition in \cite{honnappa2015}. 
\end{remark}

 \subsection{Appointment Schedules in Practice}\label{se:applying}

The proposed appointment schedule $(\bfp^{n,*}, \{a^{n,*}_i\}_{i=1}^{\bfp^{n,*}})$ depends on $n$ and the fluid limits $\bar\data.$ In practice, we may consider the FCP associated with the \emph{unscaled} system parameters. In this way, one doesn't need to define $n$ or compute $\bar\data.$
To be more precise, assume that we have the system considered in Section \ref{sec:model} with the system parameter $\data= (\mu, r, c_w, c_i, c_o, \{F(t,a);t\ge0, a\in[0,T]\})$. We consider the FCP defined in Definition \ref{df:fcp} for the system parameter $\data.$ Following \eqref{policy}, we construct $(m^*, \{a^*_i\}_{i=1}^{m^*})$ using the optimal solution of the FCP associated with $\data$. More precisely, $m^* = \lfloor A^*(T) \rfloor$ and $a^*_i = \inf\{t\in[0,T]: A^*(t)/A^*(T)\ge i/m^*\}$ for $i=1, \ldots, m^*.$ In the next Section \ref{sec:QP}, we numerically solve the FCP by developing a discrete quadratic programming. 

We now evaluate the performance of $(m^*, \{a_i\}_{i=1}^{m^*})$. Assume that $\mu$ is large, and $r, c_w$ are much smaller than $c_i$ and $c_o$ (satisfying Assumptions \ref{htc} and \ref{cost_para}). We first derive the following linearity property of the FCP.

\begin{proposition}[Linearity]\label{scale}
Let $\data^\kappa = (\kappa\mu, r, c_w, \kappa c_i, \kappa c_o, \{F(t,a);t\ge0, a\in[0,T]\})$, where $\kappa>0$. If $A^*(t)$ is an optimal solution to the FCP associated with $\data= (\mu, r, c_w, c_i, c_o, \{F(t,a);t\ge0, a\in[0,T]\})$, then $\kappa A^*(t)$ is an optimal solution to the FCP associated with $\data^\kappa$ and $J^*(\data^\kappa) = \kappa J^*(\data)$.
 \end{proposition}

From the above Proposition \ref{scale}, we have
\begin{align}\label{eq1}
J^*(\data) = \mu J^*(\data^{\mu^{-1}}),
\end{align}
where $\data^{\mu^{-1}} = (1, r, c_w, c_i/\mu, c_o/\mu,  \{F(t,a);t\ge0, a\in[0,T]\}).$ Now let's think of $\mu$ as the scaling parameter, i.e., $n=\mu$ in the asymptotic framework, and recall that $\mathcal{J}(\bfp^*, \{a^{*}_i\}_{i=1}^{\bfp^{*}}; \data)$ denotes the objective function of the stochastic QCP. From \eqref{optimal_2}, we expect
\begin{align}\label{eq2}
\mathcal{J}(\bfp^*, \{a^{*}_i\}_{i=1}^{\bfp^{*}}; \data) = \mu J^*(\bar\data) +  o\left(\mu\right),
\end{align}
where  $\bar\data \approx \data^{\mu^{-1}} = (1, r, c_w, c_i/\mu, c_o/\mu,  \{F(t,a);t\ge0, a\in[0,T]\}).$ 
Combining \eqref{eq1} and \eqref{eq2} yields
\begin{align}\label{eq3}
\mathcal{J}(\bfp^*, \{a^{*}_i\}_{i=1}^{\bfp^{*}};\data) = \mu J^*(\data^{\mu^{-1}}) +  o\left(\mu\right) = J^*(\data) +  o\left(\mu\right).
\end{align}
Now from \eqref{optimal_1}, the FCP provides an asymptotic upper bound for the original QCP, i.e., for any appointment schedule $(m, \{a_i\}_{i=1}^m)$,
\begin{align}\label{eq4}
\mathcal{J}(m, \{a_i\}_{i=1}^{m}; \data) \le \mu J^*(\data^{\mu^{-1}}) +  o\left(\mu\right) = J^*(\data) + o(\mu).
\end{align}
From \eqref{eq3} and \eqref{eq4}, we conclude that 
\begin{align}\label{eq5}
\mathcal{J}(\bfp^*, \{a^{*}_i\}_{i=1}^{\bfp^{*}};\data) \gtrapprox \mathcal{J}(m, \{a_i\}_{i=1}^{m}; \data)
\end{align}
with an error order of $o(\mu)$. Therefore, the appointment schedule $(\bfp^*, \{a^{*}_i\}_{i=1}^{\bfp^{*}})$ is near-optimal with the error of $o(\mu).$

\section{Quadratic Programming Formulation} \label{sec:QP}

In this section we discretize the FCP as a quadratic programming (QP). We discretize the time horizon $[0,T]$ into $K$ segments, with end points $t_0, t_1, \cdots, t_K$, where  $t_k = {kT}/{K}$ for $k=0, 1,\dots, K$. We call $K$ the resolution of the discretization. Let $a_k = A(t_{k+1})-A(t_k)$, $k=0,1,\cdots,K-1$, and define
\[ \mathbf{a} = (a_0, a_1, \cdots, a_{K-1})^T.\]
 We replace the continuous convolution $H(t)$ by its discrete version
$$
\hat{H}(t) = \sum_{k=1}^{K} a_{k-1} F(t-t_{k-1}, t_{k-1}) = \boldsymbol{F}(t)^T \boldsymbol{a},
$$
where $\boldsymbol{F}(t) = (F(t-t_{0}, t_{0}), F(t-t_{1}, t_{1}), \ldots, F(t-t_{K-1}, t_{K-1}))^T.$
 We also discretize the Skorokhod map on $[0,T]$ as follows: For $k=0, 1, \dots, K$,
$$
\hat{i}(k) = \frac{1}{\mu} \displaystyle\max_{i \in \{0, 1, \dots, k \}}\max \{ \mu t_i - \hat{H}(t_i), 0 \},
$$
and
$$
\hat{q}(k) = \hat{H}(t_k) - \mu t_k + \mu \hat{i}(t_k).
$$

\begin{lemma}\label{lm:lindley_skorokhod}
 Let $\Delta \hat i(k) = \hat i(k) - \hat i(k-1)$ representing the idle time over the period $(t_{k-1}, t_k]$. Then $(\hat q(k), \Delta \hat i(k))$ satisfies the discrete Lindley's recursion
  \begin{equation}\label{qp:lindley}
    \begin{cases}
      \hat q(k+1) = \max \left\{0, \hat q(k) + \hat{H}(t_k) - \hat{H}(t_{k-1}) - \frac{\mu T}{K}  \right\} \ \text{for} \ k=0, \dots, K-1 \\
     \Delta \hat i(k) =\frac{1}{\mu} \max \left\{0, \frac{\mu T}{K} - (\hat{H}(t_k) - \hat{H}(t_{k-1})) -\hat q(k)   \right\} \ \text{for} \ k=1, \dots, K. &
    \end{cases}
  \end{equation}
%  with $q_0 = \hat{H}(0)$ satisfies the discretized Skorokhod map on $[0, T]$; i.e.,
%  $q_k = \hat{q}(k)$ and $\displaystyle \sum_{j=0}^{k} i_j = \hat{i}(k)$ for $k=0, 1, \dots, K$.
\end{lemma}

%Note that $q_k$ corresponds to the fluid queue length $q(t)$ at time $t_k$ and $i_k$ corresponds to the idle time over the period $(t_{k-1}, t_k]$.

Using the discretized convolution and Lindley's recursion for the Skorokhod map, we are able to discretize the FCP as the following QP. Denote by $\hat\data = (\mu, r, c_w, c_i, c_o, \{\mathbf{F}(t_i), i = 1, \ldots, K\})$ the system parameter. 
\begin{definition}[QP]
The QP is to select $\mathbf{a}$ to maximize
\begin{equation}\label{qp:quad_prog}
  \begin{array}{rl}
    \displaystyle \displaystyle \hat{J}_K(\mathbf{a}; \hat\data) & =  r \displaystyle \sum_{k=0}^{K-1} F(T-t_k, t_k)a_k - \frac{c_w}{2 \mu} \hat q(K)^2 - \frac{c_o}{\mu} \hat q(K) - \displaystyle \sum_{k=0}^{K-1} \left( \frac{c_w T}{K} \hat q(k) + c_i \Delta \hat i(k)  \right) \\
    \text{subject to} & \Delta \hat i(0) = 0, \ \  \hat q(0) = \displaystyle \sum_{k=0}^{K-1} F(-t_k, t_k)a_k, \\
    & \hat q(k) = \hat q({k-1}) - \frac{\mu T}{K}  + \displaystyle\sum_{i=0}^{K-1} a_i \left( F(t_k - t_i, t_i) - F(t_{k-1} - t_i, t_i) \right), \ k=1, \dots, K, \\
    & \Delta \hat i(k) = \frac{T}{K} - \frac{1}{\mu}\hat q({k-1}) - \frac{1}{\mu} \displaystyle\sum_{i=0}^{K-1} a_i \left( F(t_k - t_i, t_i) - F(t_{k-1} - t_i, t_i) \right), \ k=1, \dots, K, \\
    & \Delta \hat i(k) \geq 0, \ \hat q(k) \geq 0, \ \ k=0, \dots, K.
  \end{array}
\end{equation}
\end{definition}

This QP has a linear number of variables and constraints in the resolution $K$. The quadratic component comes from the wait cost of $\hat q(K)$ patients remaining in queue at time $T$ as the queue length decreases linearly to $0$ at rate $\mu$. It is important to note that the QP's optimal objective converges to the FCP's optimal objective as $K \rightarrow \infty$ which can be observed by examining %Lemma \ref{lm:riemann_sum_error1} and 
Proposition \ref{th:qp_error_bound}.

% For $t \in [0, T]$, let $G_t(s) = F(t-s,s), s\in[0, T].$ 
% \begin{assumption}\label{assum:qp:convergence}
% For each $t$, $G_t(s)$ is Riemann–Stieltjes integrable with respect to any feasible solution $A(s)$ on $[0, T].$  
% \end{assumption}
Denote $J^*(\mathcal{M}) = \sup_A J(A; \mathcal{M})$ and $\hat J^*_K(\hat{\mathcal{M}})$ which are the optimal values of the FCP and the associated QP. 
%Recall in Assumption \ref{unp_dist}, the distrition $F(t,a)$ is piecewise continuous in $a$ with respect to the finite partition $0=\tau_0 < \cdots < \tau_L=T$. We refine our current discretization by adding $L-1$ points $\tau_1, \ldots, \tau_{L-1}.$
\begin{proposition}\label{th:qp_error_bound}
Assuming $\sup_{t\in [0, T]}|H(t) - \hat H(t)|\to 0$ as $K\to\infty$ for each feasible control $A(t)$ and the corresponding discretized $\mathbf{a}$, then we have  $|J(A; \mathcal{M})-\hat J_K(\mathbf{a}; \hat{\mathcal{M}})| \to 0$ and $|J^*(\mathcal{M}) - \hat J_K^*(\hat{\mathcal{M}})| \to 0$ as $K\to\infty.$ 
\end{proposition}

\begin{remark}
We introduce two sufficient conditions for ensuring $\sup_{t\in [0, T]}|H(t) - \hat H(t)|\to 0$ as $K\to\infty$ for each feasible control $A(t)$ in Proposition \ref{th:qp_error_bound}. These conditions are adapted from the Riemann-Stieltjes integrability conditions (cf. \cite{rudin1964principles}). 
\begin{itemize}
    \item[\rm (i)] For each $t$, if $F(t-s,s)$ is continuous in $s$ and $A(s)$ is of bounded variation on $[0, T]$. 
    \item[\rm (ii)] For each $t$, if $F(t-s,s)$ is bounded with a finite number of discontinuous points, and $A(t)$ is of bounded variation and continuous at all discontinuous points on $[0, T]$.
\end{itemize}
\end{remark}

\section{Numerical Results}\label{sec:numerical}
We present our numerical results in two sections. The first section will focus on initial insights with respect to two questions:
\begin{enumerate}
  \item How do appointment profiles change relative to different types of unpunctuality distributions?
  \item How do these appointment schedules change as we incorporate time-heterogeneous unpunctuality?
\end{enumerate}

The second section of the numerical results will look at a real dataset from several clinics. In particular, we will look at a single doctor who sees a large number of patients in a single day (at least 60 patients), as we wish to motivate the asymptotically optimal schedule as high performing for systems in which a large number of patients need to be seen. The dataset provides us with actual appointment times throughout a day and each is associated with an actual unpunctuality for the patient provided via the check-in time of the patient and its difference from the scheduled appointment time. Reliable service times are not provided, so we consider three parametric cases for service time distribution: deterministic, exponential, and log-normal distributions.

\subsection{Preliminary Insights}
We examine the optimal appointment schedules derived from various parametric unpunctuality distributions. We consider three classes of parametric distributions in this section: Uniform on $(a,b)$, denoted $U(a,b)$, Normal with mean $\mu$ and standard deviation $\sigma$, denoted $N(\mu, \sigma)$, and Generalized Laplace, denoted $\mathbb{L}(\mu, \pi, \lambda_l, \lambda_r).$ The first two have well-known density curves. The generalized Laplace distribution has the following mixture density curve:
$$
f(x; \mu, \pi, \lambda_l, \lambda_r) = \begin{cases}
                                         \pi \lambda_l e^{\lambda_l(x-\mu)}, & \mbox{if } x \leq \mu \\
                                         (1-\pi) \lambda_r e^{-\lambda_r (x-\mu)}, & \mbox{if } x > \mu.
                                       \end{cases}
$$
For all appointment schedules in this subsection, we use $c_w = 1$, $c_i = 50$, $c_o=75$, $r=0$, $\mu=100$, and $T = 1$. These values reflect our heavy traffic assumptions. For the numerical discretization, we use $K = 1000$.

We will use the following sets of mean-variance pairs across the three distributions: $(\mu, \sigma^2)$ of $(-0.1, 0.0025), (-0.05, 0.01), (0, 0.04)$. From left to right, the mean unpunctuality is increasing along with the variance. This means, for normally distributed unpunctuality, we have the following distributions: $N(-0.1, 0.0025)$, $N(-0.05, 0.01)$, and $N(0, 0.04)$. Figure \ref{plot:uniform_normal_insights}'s left plot shows the appointment profile $A(t)$ over $[0,T]$ outputted by our quadratic program with the normal unpunctuality distributions. It is interesting to note that normal unpunctuality produces block-schedules. Upon close examination, we saw that $\sum_{i=0}^{K-1} p_i f(x-t_i; \mu, \sigma)$, where $f(x; \mu, \sigma)$ is the density curve for a $N(\mu, \sigma)$ distribution, would approximate a uniform-type density better and better for smaller $\sigma$ values. Thus, the solution to our quadratic program acted as a series of mixture densities for a normal series approximation to a uniform distribution.

For the uniform unpunctuality case, we consider 3 distributions: $U(-0.1866, -0.0134)$, $U(-0.2232, 0.1232)$, and $U(-.3464, 0.3464)$. These parameters ensure that our mean and variances match the normal cases. Figure \ref{plot:uniform_normal_insights}'s right plot shows the appointment profile $A(t)$ over $[0,T]$ outputted by our quadratic program with the uniform unpunctuality distributions.

\begin{center}
  \begin{figure}
    \centering
    \begin{tabular}{|c|c|}
      \hline
      \includegraphics[width=7.5cm]{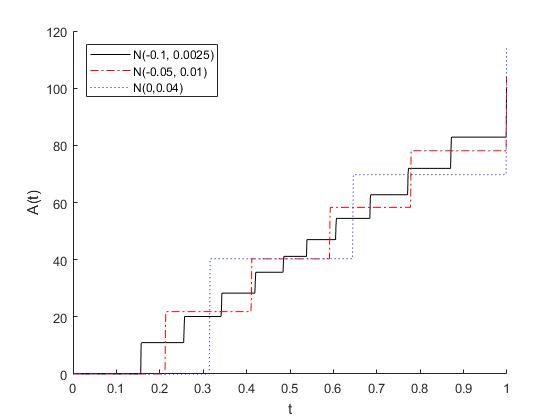} & \includegraphics[width=7.5cm]{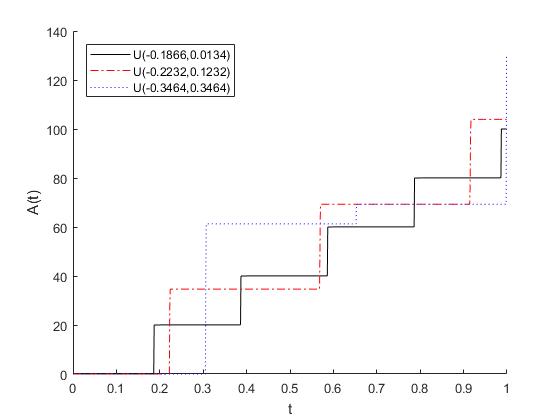} \\
      \hline
    \end{tabular}
    \caption{Left: appointment profiles derived for 3 different Normal unpunctuality distributions. Right: appointment profiles derived for 3 different Uniform unpunctuality distributions.}\label{plot:uniform_normal_insights}
  \end{figure}
\end{center}

For the generalized Laplace distribution, we consider three distributions. $\mathbb{L}(-0.1211, 0.35, 45, 22.5)$,  $\mathbb{L}(-0.05, 0.5, 14.15, 14.15)$, and $\mathbb{L}(0.085, 0.65, 5.59, 11.18)$. These parameters are once again chosen to match the same mean and variance as in the uniform and normal distribution scenarios; however, the generalized Laplace distribution allows for skewness to be incorporated. Figure \ref{plot:GLP_insights}'s left plot shows the PDFs associated with each of the parameter cases, whereas the right plot shows the associated appointment profiles $A(t)$ derived from our quadratic program.
\begin{center}
  \begin{figure}
    \centering
    \begin{tabular}{|c|c|}
      \hline
      \includegraphics[width=7.5cm]{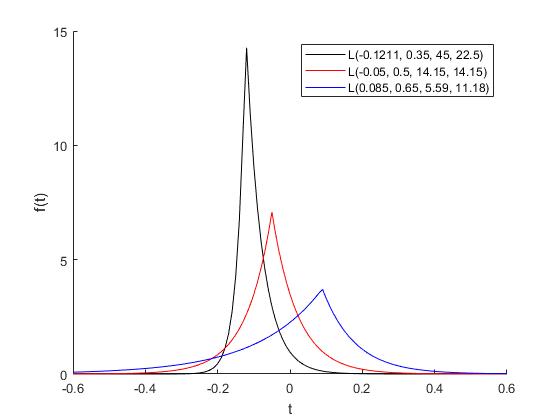} & \includegraphics[width=7.5cm]{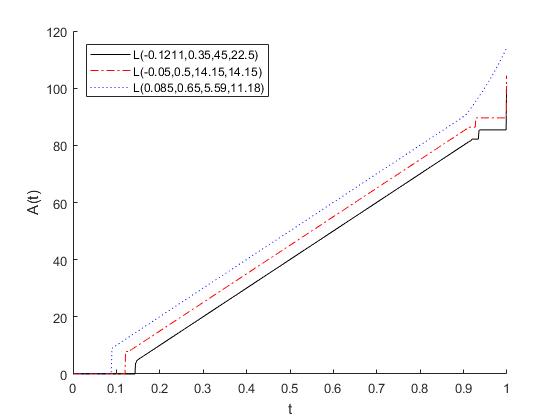} \\
      \hline
    \end{tabular}
    \caption{Left: PDFs for 3 different generalized Laplace unpunctuality distributions. Right: appointment profiles derived for the same 3 generalized Laplace unpunctuality distributions.}\label{plot:GLP_insights}
  \end{figure}
\end{center}
It is interesting to note that the generalized Laplace distribution leads to appointment schedules that appear closer to uniform in distribution plus a shift relative to the expected unpunctuality. For example, the $\mathbb{L}(-0.1211, 0.35, 45, 22.5)$ distribution has a high probability of arriving early and leads to appointments not being scheduled until around time 0.15 with a small block of patients to start; the schedule increases relatively uniformly until a block schedule at the end of the appointment profile.

We also considered the impact of time-heterogeneity when scheduling and how it impacted the appointment profiles. We consider two types of time-heterogeneity: a midday split scenario, which has a single distribution for the first half of the scheduling horizon and a different distribution for the second half of the horizon, and a parametric drift scenario where a single parametric family is picked for the distribution of unpunctuality, but its parameters change continuously over the time horizon $[0,1]$.

For the midday split, we use the following unpunctuality distribution:
$$
U_i | a_i = \begin{cases}
        U_{i,e} \sim N(0,0.04), & \mbox{if } a_i \leq \frac{1}{2} \\
        U_{i,l} \sim N(-0.1, 0.0025), & \mbox{otherwise}.
      \end{cases}
$$
This causes patients scheduled in the earlier half of the horizon to arrive on time, on average, but with greater variance. If they are scheduled later in the day, they arrive earlier on average with less variance. These parameter choices match two of the earlier scenarios.

For the parametric drift, we have
$$
U_i | a_i \sim N(\mu(a_i), \sigma^2(a_i)),
$$
where $\mu(t) = -0.1t$ and $\sigma(t) = 0.2-0.15t$ reflecting a constant linear drift in mean and standard deviation over the scheduling horizon.

The left plot of Figure \ref{plot:hetero_insights} shows the schedules resulting from the midday split and parametric drift normal distribution scenarios. To compare, the schedules derived from homogeneous $N(0, 0.04)$ and $N(-0.1, 0.0025)$ are presented as dotted lines for comparison.

We complete a similar analysis for the generalized Laplace distribution, with both a midday split and parametric drift scenario. The midday split is as follows:
$$
U_i | a_i = \begin{cases}
        U_{i,e} \sim L(0.085, 0.65, 5.59, 11.18), & \mbox{if } a_i \leq \frac{1}{2} \\
        U_{i,l} \sim L(-0.1211, 0.35, 45, 22.5), & \mbox{otherwise}.
      \end{cases}
$$
The parametric drift case again is a continuously evolving distribution setup as follows:
$$
U_i | a_i \sim L(\mu(a_i), \pi(a_i), \lambda_1(a_i), \lambda_2(a_i)),
$$
with $\mu(t) = 0.085 - 0.2061t$, $\pi(t) = 0.65-0.3t$, $\lambda_1(t) = 5.59+39.41t$, and $\lambda_2(t) = 11.18+11.32t$.

The left plot of Figure \ref{plot:hetero_insights} shows the schedules resulting from these two time-heterogeneous generalized Laplace scenarios with homogeneous-derived schedules for comparison.

\begin{center}
  \begin{figure}
    \centering
    \begin{tabular}{|c|c|}
      \hline
      \includegraphics[width=7.5cm]{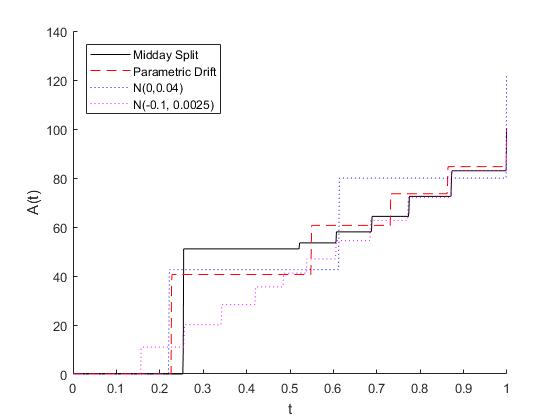} & \includegraphics[width=7.5cm]{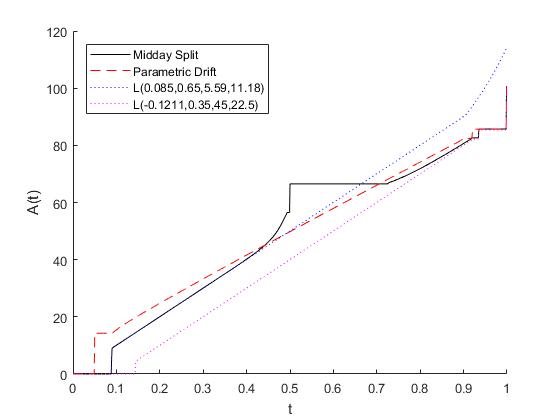} \\
      \hline
    \end{tabular}
    \caption{Left: appointment profiles derived with midday split and parametric drift time-heterogeneous unpunctuality based on Normal distributions with homogeneous-derived schedules for comparison (dotted lines). Right: appointment profiles derived with midday split and parametric drift time-heterogeneous unpunctuality based on generalized Laplace distributions with homogeneous-derived schedules for comparison (dotted lines).}\label{plot:hetero_insights}
  \end{figure}
\end{center}

It can be seen that, on either side of a midday split, the schedules reflect the respective appointment schedules from the reference distributions with the dotted lines. In the parametric drift scenarios, the schedules vary their behaviour throughout the day: in the normal case, the blocks grow closer and closer together as the variance decreases in time; in the generalized Laplace case, the slope of the schedule appears to be shallower than the reference schedules; however, the schedule starts with a sizeable block near the beginning to make up for it.

While we considered the different types of unpunctuality distributions and how their shapes affect appointment schedules. However, these results come from a zero reward assumption. As rewards are increased from zero, we expect overbooking to increase.

We consider two distributions for unpunctuality with variable rewards: a $N(-0.05, 0.01)$ distribution and a $L(-0.1211, 0.35, 45, 22.5)$ distribution. We keep all cost parameters except for $r$ fixed with $c_w = 1$, $c_i=50$, and $c_o=75$. The service rate is again $\mu =100$. We consider rewards from the set $\{0, 1, 1.2, 1.5\}$. Figure \ref{plot:reward_insights} shows the different booking structures under these reward structures.

\begin{center}
  \begin{figure}
    \centering
    \begin{tabular}{|c|c|}
      \hline
      \includegraphics[width=7.5cm]{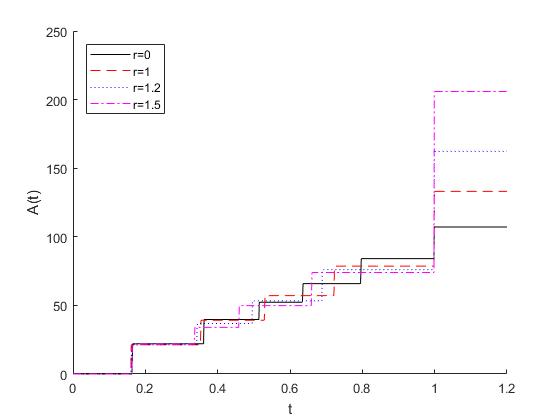} & \includegraphics[width=7.5cm]{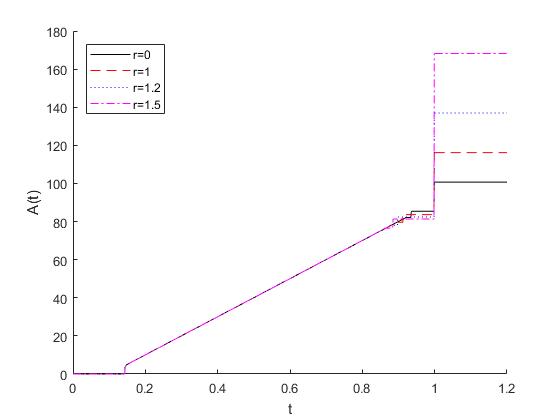} \\
      \hline
    \end{tabular}
    \caption{Left: Appointment bookings under different reward structures with $N(-0.05, 0.01)$ unpunctuality. Right: Appointment bookings under different reward structures with $L(-0.1211, 0.35, 45, 22.5)$ unpunctuality.}\label{plot:reward_insights}
  \end{figure}
\end{center}

\subsection{Data-driven Experiments}
We examined a new dataset that contains patient unpunctuality and appointment schedules. The data is across several clinics and each data point corresponds to a single appointment. An appointment includes anonymized clinic ID, anonymized doctor ID, in addition to scheduled appointment times and actual arrival times, from which we can derive the unpunctuality associated with the appointment. We examined the full data in addition to data for the single clinic and single doctor with the highest booking frequency. Figure \ref{plot:emp_insights} presents the empirical CDFs for unpunctuality across the 3 different subsets of the data.

\begin{center}
  \begin{figure}
    \centering
    \begin{tabular}{|c|c|}
      \hline
      \includegraphics[width=7.5cm]{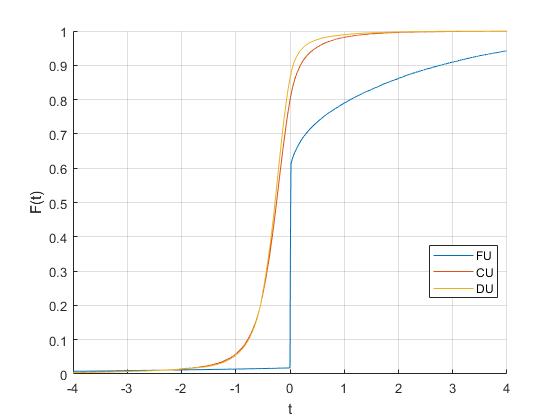} & \includegraphics[width=7.5cm]{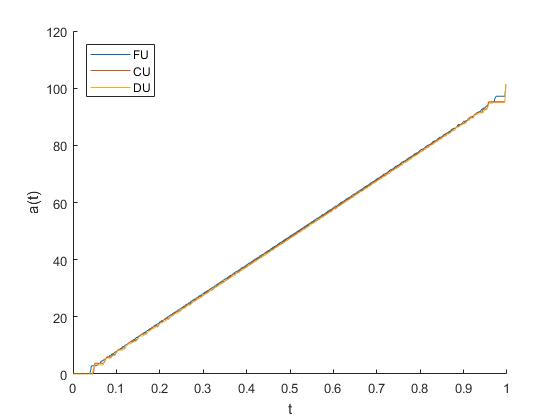} \\
      \hline
    \end{tabular}
    \caption{Left: Empirical CDFs for full data (FU), largest-booking clinic data (CU), and largest-booking doctor data (DU) Right: appointment profiles derived for the same 3 unpunctuality distributions. Unpunctuality is rescaled such that $T=1$.}\label{plot:emp_insights}
  \end{figure}
\end{center}

We now examine how a schedule derived from our quadratic program performs relative to a baseline in a discrete event simulation setting. As we have access to the scheduled appointment times in our data, we can set our baseline to be the actual schedules used by these clinics in the past. Each system will be a single-server queue following a FIFO service discipline. Due to this system setup, given an appointment schedule and a sample path of unpunctuality and service time values, we can easily calculate the components of the objective function using Lindley's equations. Sample paths for unpunctuality can be drawn from the data and service times will be generated as either deterministic, exponential, or log-normal random variables. Looking at the largest-booking doctor, we can examine a particular day. However, the scheduling window of a day is not consistent across several days. To simplify things, we normalize a day's appointment schedule so it falls within the range $[0,1]$ - this is based off the desire to set $T=1$. Suppose there are $P$ patients scheduled for a single day and $a_i$ is the time of the appointment for the $i$th patient for $i=1, \dots, P$. We normalize $a_i$ as follows:
$$
\tilde{a}_i = \frac{a_i - \displaystyle \min_{i=1,\dots,P} \{a_i\}}{\displaystyle \max_{i=1,\dots,P} \{a_i\} - \displaystyle \min_{i=1,\dots,P} \{a_i\}}
$$
As appointments are being rescaled, we must similarly rescale unpunctuality. If $u_i$ is the unpunctuality of the $i$th patients for $i=1, \dots, P$, we rescale $u_i$ as follows:
$$
\tilde{u}_i = \frac{u_i}{\displaystyle \max_{i=1,\dots,P} \{a_i\} - \displaystyle \min_{i=1,\dots,P} \{a_i\}}.
$$

We have several days of data for the highest-booking doctor. However, not all days have the same number of patients, with some days having very few bookings. We wish to investigate how our QP-based appointment scheduling compares to the existing schedules in the data. As our model is based off an asymptotically optimal formulation, we will consider only days where the number of bookings is on the large-side: the day must have at least 60 patients. This leaves us with 492 days to compare schedules to. Another issue to note is that, in the data, the number of patients seen on a particular day varies greatly; however, we have no indication of the cost structure used to derive these schedules. To address this, we keep $r=0, c_w=1$ and vary $c_i$ while keeping $c_o = 1.5c_i$ fixed. We choose $c_i$ within the set $\{50, 75, 100, 150 \}$. Also, we only need to solve a single quadratic program to produce an appointment schedule as we employ scalar multiplication of $P$ (see Proposition \ref{scale}) in order to match the correct number of patients to be scheduled with data on a particular day.

As correct values for service times are not provided by the data, we instead simulate service times. Let $P$ be the number of patients scheduled on a particular day. Service times are drawn from three different distributions:
\begin{enumerate}
  \item $\delta_{{1}/{P}}$: service times are deterministic with fixed value $\frac{1}{P}$.
  \item $Exp(P)$: service times are exponential with rate $P$.
  \item $Lognormal(\mu, 2)$: service times are log-normally distributed with logarithmic standard deviation $\sigma = 2$ and logarithmic mean $\mu = - \log(P) - \frac{\sigma^2}{2}$.
\end{enumerate}
The parameters of these distributions are chosen such that the expected value of the sum of the service times is equal to $T = 1$, meaning the clinic is neither overbooked nor underbooked relative to the service time durations.

Because we let $r=0$, the control problem in \eqref{scp} is equivalent to minimize the cost function consisting of holding cost, idleness cost and overtime cost. In the following Tables \ref{tab:QPvsData_mean_CIs_STdet}, \ref{tab:QPvsData_mean_CIs_STexp} and \ref{tab:QPvsData_mean_CIs_STlogN}, the objective values represent the total cost under the corresponding schedule control and parameter value of $c_i$. Table \ref{tab:QPvsData_mean_CIs_STdet} provides 95\% Bonferroni simulteanous confidence intervals for the mean value of the objective function for each schedule across the 492 days with deterministic service times. We compare several schedules:
\begin{itemize}
  \item Actual refers to using the actual schedule recorded in the dataset.
  \item ZU refers to an optimal schedule under zero unpunctuality according to Proposition \ref{fcp:no:pun}.
  \item QP refers to a data-driven schedule derived from using the quadratic program with an empirical unpunctuality distribution.
\end{itemize}
The tables also provide relative improvements on the actual schedule, where each relative improvement is for a single day, measured pairwise between the actual and proposed schedule across $n=492$ days.

\begin{table}[h!]
\centering
\begin{tabular}{||c | c | c c | c c||}
 \hline
 $c_i$ & Actual & ZU & ZU Rel. Imp. & QP & QP Rel. Imp. \\ [0.5ex]
 \hline\hline
$ 50$ & $36.44 \pm 0.96$ & $34.32 \pm 0.90$ & $5.63 \pm 0.61$ & $31.86 \pm 12.55$ & $12.55 \pm 0.62$ \\
 $75$ & $37.39 \pm 0.94$ & $35.36 \pm 0.88$ & $5.24 \pm 0.59$ & $33.02 \pm 0.86$ & $11.63 \pm 0.64$ \\
$ 100$ & $38.33 \pm 0.92$ & $36.40 \pm 0.87$ & $4.86 \pm 0.59$ & $34.19 \pm 0.85$ & $10.76 \pm 0.67$ \\
$ 150$ & $40.22 \pm 0.92$ & $38.49 \pm 0.87$ & $4.17 \pm 0.61$ & $36.51 \pm 0.86$ & $9.17 \pm 0.74$ \\ [0.5ex]
 \hline
\end{tabular}
\caption{95\% mean confidence intervals of objective value and relative improvement across different cost structures with deterministic service times.}
\label{tab:QPvsData_mean_CIs_STdet}
\end{table}

Table \ref{tab:QPvsData_mean_CIs_STexp} provides the same results for the case that service times are exponentially distributed. We notice that switching from deterministic to exponential does not seem to have a strong impact on the mean values of the objective function for the different cost structures.

\begin{table}[h!]
\centering
\begin{tabular}{||c | c | c c | c c||}
 \hline
 $c_i$ & Actual & ZU & ZU Rel. Imp. & QP & QP Rel. Imp. \\ [0.5ex]
 \hline\hline
 $50$ & $40.01 \pm 1.28$ & $37.79 \pm 1.20$ & $5.31 \pm 0.63$ & $35.16 \pm 1.17$ & $12.15 \pm 0.66$ \\
 $75$ & $42.59 \pm 1.37$ & $40.40 \pm 1.28$ & $4.86 \pm 0.67$ & $37.78 \pm 1.25$ & $11.26 \pm 0.70$ \\
 $100$ & $45.16 \pm 1.48$ & $43.01 \pm 1.39$ & $4.46 \pm 0.72$ & $40.40 \pm 1.36$ & $10.49 \pm 0.76$ \\
 $150$ & $50.32 \pm 1.78$ & $48.24 \pm 1.69$ & $3.78 \pm 0.84$ & $45.65 \pm 1.67$ & $9.21 \pm 0.90$ \\ [0.5ex]
 \hline
\end{tabular}
\caption{95\% mean confidence intervals of objective value and relative improvement across different cost structures with exponential service times.}
\label{tab:QPvsData_mean_CIs_STexp}
\end{table}

Table \ref{tab:QPvsData_mean_CIs_STlogN} similarly provides the same simulation results, but with log-normal simulated service times. In this case, we see that the log-normal distribution leads to a more noticeable increase in objective function as opposed to switching from deterministic to exponential.

\begin{table}[h!]
\centering
\begin{tabular}{||c | c | c c | c c||}
 \hline
 $c_i$ & Actual & ZU & ZU Rel. Imp. & QP & QP Rel. Imp. \\ [0.5ex]
 \hline\hline
 $50$ & $60.69 \pm 12.22$ & $59.03 \pm 12.03$ & $3.13 \pm 0.64$ & $57.14 \pm 12.00$ & $6.93 \pm 0.89$ \\
 $75$ & $72.74 \pm 15.15$ & $71.24 \pm 14.96$ & $2.43 \pm 0.69$ & $69.51 \pm 14.93$ & $5.32 \pm 1.02$ \\
 $100$ & $84.79 \pm 18.09$ & $83.45 \pm 17.90$ & $1.93 \pm 0.74$ & $81.88 \pm 17.88$ & $4.14 \pm 1.13$ \\
 $150$ & $108.88 \pm 24.00$ & $107.87 \pm 23.82$ & $1.23 \pm 0.82$ & $106.61 \pm 23.80$ & $2.47 \pm 1.31$ \\ [0.5ex]
 \hline
\end{tabular}
\caption{95\% mean confidence intervals of objective value and relative improvement across different cost structures with log-normal service times.}
\label{tab:QPvsData_mean_CIs_STlogN}
\end{table}

We can see that, in general, the incorporation of unpunctuality will produce the best results whereas a zero unpunctuality assumption performs only slightly better than, or on par with, the actual schedules.

\section{Conclusion}\label{sec:Conclusion}

In this paper, we developed a framework for computing asymptotically optimal appointment schedules with generalized patient unpunctuality that also considers the under/overbooking of patients. As the general stochastic control problem is not analytically solvable, we developed an approximate deterministic FCP and a time-discretized QP to numerically solve the FCP. We showed that as the number of patients approaches infinity, the FCP provides an asymptotically optimal policy for the original stochastic control problem, and as the time-discretization grows to zero, the QP approaches the FCP. We then finished with a numerical study that first examined the appointment profiles associated with several different types of unpunctuality. It was interesting to note that even in the normal case, multi-block-type schedules were optimal under our model. However, other types of unpunctuality, such as generalized Laplace, may lead to more fluid-arrival-type appointment profiles that appear shifted relative to expected unpunctuality. Further, we discovered that time-heterogeneous unpunctuality can lead to significant differences in the character of appointment schedules from the homogeneous case. In our final section of the numerical studies, we compared schedules derived from the FCP or QP against a real dataset with real schedules and unpunctuality values across several hundred days. We saw that our QP model that incorporates unpunctuality has the strongest performance compared to the fluid model without unpunctuality and the actual appointment schedules themselves.

\bibliographystyle{plain}
\bibliography{ref.bib}

\appendix

\section{Mathematical Pre-requisites}

We introduce the one-dimensional Skorokhod problem based on the lecture notes of \cite{williams2017}. First, let $\mathbb{D}_0$ be the set of all RCLL functions $x: [0, \infty) \rightarrow \mathbb{R}$ with the restriction that $x(0) \geq 0$, and $\mathbb{D}_+$ be the set of all RCLL functions $y: [0, \infty) \rightarrow \mathbb{R}_+$ such that $y(t)\ge 0$ for each $t\ge 0$.

\begin{definition}[One-dimensional Skorokhod Problem]\label{sm:def}
  Let $x \in \mathbb{D}_0$. A pair of functions $(z, y) \in \mathbb{D}_+ \times \mathbb{D}_+$ is a solution of the one-dimensional Skorokhod problem for $x$ if the following conditions holds:
  \begin{enumerate}
    \item $z(t) = x(t) + y(t), \ \ t \geq 0$,
    \item $z(t) \geq 0, \ \ t \geq 0$,
    \item $y$ satisfies the following: \begin{enumerate}
                                         \item $y(0) = 0$,
                                         \item $y$ is nondecreasing,
                                         \item $\int_{0}^{\infty} z(t) dy(t) = 0$.
                                       \end{enumerate}
  \end{enumerate}
\end{definition}
Note: $(z,y)$ is called a regulation of $x$ where $y$ behaves as the regulator and $z$ is the regulated path. The regulated path is made nonnegative and is comparable to the queuing process from our fluid model. The regulator ensures the boundary reflection property of the regulated path and is comparable to the $\mu i(t)$ process from our fluid model. The $x(t)$ is comparable to $H(t) - \mu t$ in our model.

\cite{williams2017} next proves that for each $x \in \mathbb{D}_0$, there exists a unique solution to the Skorokhod problem, defining a Skorokhod map $(\Phi, \Psi): \mathbb{D}_0 \rightarrow \mathbb{D}_+ \times \mathbb{D}_+$ such that $(\Phi(x), \Psi(x)) = (z,y)$, where $(z,y)$ is the unique solution of the Skorokhod problem for $x$. The theorem is formally stated as follows:
\begin{theorem}\label{sm:repres} 
  Let $x \in \mathbb{D}_0$. Then there exists a unique solution $(z,y) \in \mathbb{D}_+ \times \mathbb{D}_+$ of the Skorokhod problem for $x$ given by
  \begin{equation}\label{eqn:skorokhod_map}
    \begin{aligned}
      y(t) & =   \sup_{0 \leq s \leq t} \max \{ -x(s), 0 \}, & t \geq 0, \\
      z(t) & =  x(t) + y(t), & t \geq 0.
    \end{aligned}
  \end{equation}
  Further, if $x$ is continuous, then both $y$ and $z$ are continuous. Finally, the Skorokhod map is Lipschitz continuous in that there exists a $\kappa_0 > 0$ such that for any $x_1, x_2 \in \mathbb{D}_0$ and $T>0$, 
  \begin{align*}
      & \sup_{t\in[0, T]}|\Phi(x_1)(t)-\Phi(x_2)(t)| \le \kappa_0 \sup_{t\in[0,T]}|x_1(t)-x_2(t)|, \\
      & \sup_{t\in[0, T]}|\Psi(x_1)(t)-\Psi(x_2)(t)| \le \kappa_0 \sup_{t\in[0,T]}|x_1(t)-x_2(t)|.
  \end{align*}
\end{theorem}

% For one of our numerical proofs, we require Theorem 5 from \cite{dragomir2000} as stated below:
% \begin{theorem}\label{th:dragomirthm5}
%    Suppose $u: [a, b] \rightarrow \mathbb{R}$ be a mapping of bounded variation on $[a, b]$ and $f: [a,b] \rightarrow \mathbb{R}$ be a p-H-H{\"o}lder mapping; i.e., $ |f(x) - f(y)| \leq H|x-y|^p \ \forall x, y \in [a,b] $. Let $I_n: a=x_0 < x_1 < \dots < x_{n-1} < x_n = b$ be a partition of the interval $[a,b]$, $h_i = x_{i+1} - x_i$ ($i = 0, \dots, n-1$), $\xi_i \in [x_i, x_{i+1}]$ ($i = 0, \dots, n-1$), and $\nu (h) = \max\{ h_i | i = 0, \dots, n-1 \}$.
%    Then
%   $$
%   \displaystyle\int_{a}^{b} f(t) du(t) = \displaystyle \sum_{i=0}^{n-1} f(\xi_i)[u(x_{i+1}) - u(x_i)] + R(f, u, I_n, \xi)
%   $$
%   where
%   $$
%   \begin{array}{rcl}
%     R(f, u, I_n, \xi) & \leq & H \left[ \frac{\nu (h)}{2} + \displaystyle\max_{i=0, \dots, n-1} \left| \xi_i - \frac{x_i + x_{i+1}}{2} \right| \right]^p V_{a}^{b}(u) \\
%     &&\\
%      & \leq & H (\nu (h))^p V_{a}^{b}(u),
%   \end{array}
%   $$
%   where $V_{a}^{b}(u)$ is the total variation of $u$ over $[a,b]$.
% \end{theorem}

\section{Mathematical Proofs}
We present all the detailed proofs in this section. 

\begin{proof}[Proof of Proposition \ref{prop:fluid-limits}]
We first show that $\{A(t), t\in \RR\}$ is a CDF. It suffices to show that $A(t)$ is right continuous. Using the Moore-Osgood theorem on exchanging limits, we have
\begin{align*}
\lim_{t \downarrow s} A(t) = \lim_{t \downarrow s}\lim_{n\to\infty} \bar A^n(t) = \lim_{n\to\infty}  \lim_{t \downarrow s} \bar A^n(t)  =  \lim_{n\to\infty} \bar A^n(s) = A(s),
\end{align*}
where the first and last equalities follow from the condition that $\bar A^n(t)\to A(t)$ for each $t$, the second equality is from Moore-Osgood theorem, and the third equality is from the right continuity of $\bar A^n(t)$. This shows that  $\{A(t), t\in \RR\}$ is a CDF.

To prove \eqref{eq:fluid-arrival}, we define
\begin{align}\label{H_n}
H^n(t) =  \frac{m^n}{n}\int_{0}^T F(t-s, s) d\bar A^n(s) = \frac{1}{n}\sum_{i=1}^{m^n} F(t- a_i^n, a^n_i), \ t\in \RR.
\end{align}
Then we have the following decomposition.
\begin{align*}
\bar E^n(t) - H(t) = [\bar E^n(t) - H^n(t)] + [H^n(t) - H(t)], \ t\in \RR.
\end{align*}
We first focus on $H^n(t) - H(t).$
Fix $\epsilon > 0$. Recall that for any CDF $\tilde F$, there exists a simple function $\tilde G$ such that $\sup_{t\in\RR}|\tilde F(t) - \tilde G(t)|\le \epsilon.$ Thus for each  $s\in [0, T]$, there exists a simple function $G(t, s) = \sum_{k=1}^{K(s)} c_k(s) 1_{B_k(s)}(t), t\in \RR$ such that $\sup_{t\in\RR} (F(t, s) - G(t, s)) \le \epsilon/2$, where $K(s)$ is a positive integer, $c_k(s), k=1, \ldots, K(s)$ are positive constants and $B_k(s), k=1, \ldots, K(s)$ are disjoint subsets of $\RR$. Now noting that $F(t, s)$ is piecewise continuous in $s$ uniformly for $t$ over each piece, there exists $\delta >0$ such that when $|s_1 - s_2| \le \delta$ and $s_1$ and $s_2$ are in any partition interval,
\[
\sup_{t\in \RR}|F(t, s_1) - G(t, s_2)| \le \sup_{t\in \RR}|F(t, s_1) - F(t, s_2)| + \sup_{t\in \RR}|F(t, s_2) - G(t, s_2)| \le  \epsilon.
\]
This yields that there exists a finite collection of $0 = s_0 < s_1 < \cdots < s_m = T$ such that for $i = 1, \ldots, m,$
\begin{align}
\sup_{s\in [s_{i-1}, s_i)}\sup_{t\in \RR}|F(t, s) - G(t, s_{i-1})| \le \epsilon.
\end{align}
%We have for each $s\in [0, T]$,
%%there exist $t_L(s) < 0$ and $t_R(s) > 0$ such that $F(t, s) < \epsilon$ for $t\le t_L(s)$ and $F(t, s) > 1-\epsilon$ for $t \ge t_T(s)$. Furthermore, for each $s\in [0, T]$,
%there exists a simple function $G(t, s) = \sum_{k=1}^K(s) c_k(s) 1_{B_k(s)}(t), t\in \RR$ such that $G(t,s) \le F(t, s)$ for each $t$ and $\sup_{t\in\RR} (F(t, s) - G(t, s)) \le \epsilon$, where $K(s)$ is a positive integer (dependent on $\epsilon$), $c_k(s), k=1, \ldots, K(s)$ are positive constants and $B_k(s), k=1, \ldots, K(s)$ are disjoint subsets of $\RR$. At last, define $F^*(t) = \sup_{s\in [0, T]} F(t, s)$. Noting that $0\le \sup_{t\in\RR}F^*(t) \le 1$, there exists a simple function $G^*(t)$ such that $G^*(t) \ge F^*(t)$ and $\sup_{t\in \RR} (G^*(t) - F^*(t)) \le \epsilon.$
%%Define
%%\[
%%G(t) = \begin{cases}
%%0, & t< t_L,\\
%%g(t), & t_L\le t < t_R, \\
%%1, & t\ge t_R.
%%\end{cases}
%%\]
Thus we have
\begin{align*}
 \sup_{t\in\RR} |H^n(t) - H(t)| & =  m\cdot \sup_{t\in\RR} \left |\int_0^T F(t-s, s) d [\bar A^n(s)-A(s)]\right| \\
& = m\cdot  \sup_{t\in\RR} \left |\sum_{i=1}^m \int_{s_{i-1}}^{s_i} F(t-s, s) d [\bar A^n(s)-A(s)]\right| \\
& =  m\cdot \sup_{t\in\RR} \left |\sum_{i=1}^m \int_{s_{i-1}}^{s_i} [F(t-s, s) - G(t-s, s_{i-1})]d [\bar A^n(s)-A(s)]\right| \\
& \quad + m\cdot\sup_{t\in\RR} \left |\sum_{i=1}^m \int_{s_{i-1}}^{s_i} G(t-s, s_{i-1})d [\bar A^n(s)-A(s)]\right|\\
%& \le \sup_{t\in \RR} \left| \int_0^T [F(t-s, s) - G(t-s, s)] d[a^n(s)-a(s)]\right| + \sup_{t\in \RR} \left| \int_0^t G(t-s,s) d[a^n(s)-a(s)]\right| \\
& \le m\cdot \sup_{t\in\RR} \sum_{i=1}^m \int_{s_{i-1}}^{s_i}\sup_{s\in[s_{i-1},s_i)} \sup_{t\in\RR}|F(t, s) - G(t, s_{i-1})| d [\bar A^n(s)+A(s)]\\
& \quad + m\cdot\sum_{i=1}^m \sum_{k=1}^{K(s_i)} c_k(s_i) \sup_{s\in [s_{i-1}, s_i)}|\bar A^n(s)-A(s)|\\
& \le  m\cdot\epsilon T (\bar A^n(T)+A(T)) +p \cdot \sum_{i=1}^m \sum_{k=1}^{K(s_i)} c_k(s_i)  \sup_{s\in[0, T]}|\bar A^n(s)-A(s)|,
\end{align*}
which yields that
\begin{align}
\lim_{n\to\infty}\sup_{t\in\RR} |H^n(t) - H(t)| & = \lim_{\epsilon\to 0}\lim_{n\to\infty}\sup_{t\in\RR} |H^n(t) - H(t)| = 0
%&  \le  \lim_{\epsilon\to 0}\lim_{n\to\infty} [\epsilon (a^n(T)+a(T)) + \left(1+\sum_{k=1}^K c_k\right) \sup_{s\in[0, T]}|a^n(s)-a(s)|] \nonumber\\
%& = \lim_{\epsilon\to 0}\epsilon (a(T)+a(T))\nonumber \\
%& = 0.
\label{eq:Hn-H}
\end{align}

We now consider
\[
\bar E^n(t) - H^n(t) = \frac{1}{n} \sum_{i=1}^{m^n} [1_{\{U_i + a_i^n \le t\}} - F(t -a_{i}^n, a^n_i)], \ t\in\RR.
\]
%Let
%\[
%\Delta_i^n(t) = 1_{\{U_i + a_{i,n} \le t\}} - F(t -a_{i}^n).
%\]
%We note that $1_{\{U_i + a_{i,n} \le t\}} - F(t -a_{i}^n), i=1, \ldots, n$ are independent with common mean $0$ and are uniformly bounded by $2$. Using the SLLN for triangular arrays, we have for each $t$,
%\begin{align}\label{eq:AnHn}
%\bar A^n(t) - H^n(t) \to 0, \ \mbox{almost surely.}
%\end{align}
For $\epsilon > 0$, we can find a partition of $\RR$: $-\infty < t_0 < t_1 < \cdots < t_{L-1} < t_L  < \infty$ such that
\[
H(t_{j+1}-) - H(t_j) \le \epsilon, \ j = 0, \ldots, L-1,
\]
where $L$ is a positive integer that depends on $\epsilon$.
For each $t\in \RR$, there exists $k=0, \ldots, L-1$ such that $t\in [t_k, t_{k+1})$. We then have
\begin{align}
&\bar E^n(t) -  H^n(t)\nonumber \\
& = \frac{1}{n} \sum_{i=1}^{m^n} [1_{\{U_i + a_{i}^n \le t\}} - F(t -a_{i}^n, a^n_i)] \nonumber\\
& \ge  \frac{1}{n} \sum_{i=1}^{m^n} [1_{\{U_i + a_{i}^n \le t_{k}\}} - F((t_{k+1} -a_{i}^n)-, a^n_i)] \nonumber\\
& = \frac{1}{n} \sum_{i=1}^{m^n} [1_{\{U_i + a_{i}^n \le t_{k}\}} - F(t_{k} -a_{i}^n, a^n_i)]  -   \frac{1}{n} \sum_{i=1}^{m^n} [F((t_{k+1} -a_{i}^n)-, a^n_i) - F(t_k - a_{i}^n, a^n_i)].\label{eq:lower}
\end{align}
On the other hand,
\begin{align}
&\bar E^n(t) - H^n(t) \nonumber\\%& = \frac{1}{n} \sum_{i=1}^n [1_{\{U_i + a_{i}^n \le t\}} - F(t -a_{i}^n)] \\
 \le & \frac{1}{n} \sum_{i=1}^{m^n} [1_{\{U_i + a_{i}^n < t_{k+1}\}} - F(t_k -a_{i}^n, a^n_i)] \nonumber\\
 = & \frac{1}{n} \sum_{i=1}^{m^n} [1_{\{U_i + a_{i}^n < t_{k+1}\}} - F((t_{k+1} -a_{i}^n)-, a^n_i)]  +   \frac{1}{n} \sum_{i=1}^{m^n} [F((t_{k+1} -a_{i}^n)-, a^n_i) - F(t_k - a_{i}^n, a^n_i)].\label{eq:upper}
\end{align}
%%%% Assume random scheduled times%%%%
%For the first terms in \eqref{eq:lower}, we write for $i=1, 2, \ldots, n,$
%\[
%d^n_i = 1_{\{U_i + a_{i}^n \le t_{k}\}} - F(t_{k} -a_{i}^n, a^n_i).
%\]
%It is clear that $\EE(d^n_i) =\EE(\EE(d^n_i|a^n_i)) = 0$. Noting that each $d^n_i$ is bounded by $2$, similar to the proof of the strong law of large numbers for independent random variables with finite fourth moment, we have for $\epsilon > 0$,
%\begin{align*}
%P\left(\frac{1}{n}\left| \sum_{i=1}^n d^n_i\right| > \epsilon \right) & \le \frac{\EE[(\sum_{i=1}^n d_i^n)^4]}{\epsilon^4n^4} =  \frac{1}{\epsilon^4n^4} \sum_{i,j,k,l} \EE(d^n_i d^n_j d^n_k d^n_l).
%\end{align*}
%We now analyze $\EE(d^n_i d^n_j d^n_k d^n_l)$. If $i=j=k=l$ or $i=j\neq k=l$, we use the estimate $\EE(d^n_i d^n_j d^n_k d^n_l) \le 16$. If one index is different from all other indices, say $i\notin \{j, k, l\}$, then
%\begin{align*}
%\EE(d^n_i d^n_j d^n_k d^n_l) & = \EE[\EE(d^n_i d^n_j d^n_k d^n_l|a^n_1, \ldots, a^n_n)]\\
%& =   \EE[\EE(d^n_i |a^n_1, \ldots, a^n_n)\EE(d^n_i d^n_j d^n_k d^n_l|a^n_1, \ldots, a^n_n)] = 0.
%\end{align*}
%It follows now
%\begin{align}
%P\left(\frac{1}{n}\left| \sum_{i=1}^n d^n_i\right| > \epsilon \right)  \le \frac{16 (n+3n(n+1))}{\epsilon^4n^4}.
%\end{align}
%Using Borel-Cantelli lemma yields that
%%%%End%%%
We note that $\{1_{\{U_i + a_{i}^n < t_{k+1}\}} - F((t_{k+1} -a_{i}^n)-, a^n_i)\}_{i=1}^{m^n}$ and $\{1_{\{U_i + a_{i,n} \le t_k\}} - F(t_k -a_{i}^n, a^n_i)\}_{i=1}^{m^n}$ are independent sequences with common mean $0$ and are uniformly bounded by $2$. From the strong law of large numbers (SLLN) for triangular arrays,  %\textcolor{red}{??? what is the reference to this? This is what I could not find. ???}). \textcolor{blue}{I found the statement in Terence Tao’s blog https://terrytao.wordpress.com/2015/10/23/275a-notes-3-the-weak-and-strong-law-of-large-numbers/. Is it good enough for a reference? I will try to find a book/paper reference. The proof is the same as that for the iid random variables. Indeed, let $S_n = \sum_{i=1}^n X^n_i$. When each $\{X^n_i\}_{i=1}^n$ is an independent sequence with mean $0$ and is uniformly bounded by a constant $c$, by Chebyshev's inequality, $P(|S_n| > n\epsilon) \le E[S_n^4]/(n^4\epsilon^4)\le \tilde c n^2/(n^4\epsilon^4) = \tilde c/(n^2\epsilon^4)$, where $E[S_n^4] \le \tilde c n^2$ follows from the independence of $X^n_i, i=1, \ldots, n$ and their common mean $0$. Thus $\sum_{n=1}^\infty P(|S_n| > n\epsilon) < \infty$. Then using Borel-Cantelli lemma, we conclude that $S_n/n \to 0$ a.s.}
\begin{align}\label{eq:11}
 \frac{1}{n} \sum_{i=1}^{m^n} [1_{\{U_i + a_{i}^n < t_{k+1}\}} - F((t_{k+1} -a_{i}^n)-, a^n_i)]  \to 0, \ \ \mbox{almost surely}.
\end{align}
Similarly, we have
\begin{align}\label{eq:12}
 \frac{1}{n} \sum_{i=1}^{m^n} [1_{\{U_i + a_{i}^n \le t_{k}\}} - F(t_{k} -a_{i}^n, a^n_i)]  \to 0, \ \ \mbox{almost surely}.
\end{align}
We next note that the same second term in \eqref{eq:upper} and \eqref{eq:lower} has the following estimate.
\begin{align}
&\left|\frac{1}{n} \sum_{i=1}^{m^n} [F((t_{k+1} -a_{i}^n)-, a^n_i) - F(t_k - a_{i}^n, a^n_i)]\right|\nonumber \\
& = \left|\frac{1}{n} \sum_{i=1}^{m^n} [\lim_{\delta\downarrow 0}F(t_{k+1} -a_{i,n}-\delta, a^n_i) - F(t_k - a_{i,n}, a^n_i)]\right| \nonumber\\
& = \left|\lim_{\delta\downarrow 0} H^n(t_{k+1}-\delta) - H^n(t_k)\right|\nonumber \\
& \le \lim_{\delta\downarrow 0} |(H^n(t_{k+1}-\delta) - H(t_{k+1}-\delta)) |+ |(H^n(t_k)- H(t_k))| + \left|\lim_{\delta\downarrow 0} H(t_{k+1}-\delta) - H(t_k)\right| \nonumber\\
& \le 2\sup_{t\in\RR}|H^n(t) - H(t)| + [H(t_{k+1}-) - H(t_k)] \nonumber\\
& \le 2\sup_{t\in\RR}|H^n(t) - H(t)| + \epsilon. \label{eq:II}
\end{align}
Using \eqref{eq:upper} and \eqref{eq:lower}, we have
\begin{align*}
&\sup_{t\in\RR}|\bar E^n(t) - H^n(t)| \\
 \le & \max_{0\le k\le L-1} \left|  \frac{1}{n} \sum_{i=1}^{m^n} [1_{\{U_i + a_{i}^n < t_{k+1}\}} - F((t_{k+1} -a_{i}^n)-, a^n_i)]  +   \frac{1}{n} \sum_{i=1}^{m^n} [F((t_{k+1} -a_{i}^n)-) - F(t_k - a_{i}^n, a^n_i)]\right| \\
& + \max_{0\le k\le L-1} \left|\frac{1}{n} \sum_{i=1}^{m^n} [1_{\{U_i + a_{i}^n \le t_{k}\}} - F((t_{k} -a_{i}^n))]  -   \frac{1}{n} \sum_{i=1}^{m^n} [F((t_{k+1} -a_{i}^n)-, a^n_i) - F(t_k - a_{i}^n, a^n_i)] \right|\\
\le & \max_{0\le k\le L-1} \left|  \frac{1}{n} \sum_{i=1}^{m^n} [1_{\{U_i + a_{i}^n < t_{k+1}\}} - F((t_{k+1} -a_{i}^n)-, a^n_i)]  \right|   \\
& + \max_{0\le k\le L-1} \left|\frac{1}{n} \sum_{i=1}^{m^n} [1_{\{U_i + a_{i}^n \le t_{k}\}} - F(t_{k} -a_{i}^n, a^n_i)]  \right|+ 4\sup_{t\in\RR}|H^n(t) - H(t)| + 2\epsilon
\end{align*}
From \eqref{eq:11}, \eqref{eq:12}, and \eqref{eq:Hn-H}, we finally have
\begin{align*}
\lim_{n\to\infty}\sup_{t\in\RR}|\bar E^n(t) - H^n(t)|  & = \lim_{\epsilon\to 0}\lim_{n\to\infty} \sup_{t\in\RR}|\bar E^n(t) - H^n(t)| \le \lim_{\epsilon \to 0} 2 \epsilon =0.
\end{align*}
This completes the proof of (i). 
We now consider the fluid scaled queue length process $\bar Q^n(t), t\in \RR$, and note that for $t\in [0,T],$
\begin{align*}
\bar Q^n(t) & = \bar E^n(t) - \bar S^n(B^n(t)) \\
& = [\bar E^n(t) - H(t)] -[\bar S^n(B^n(t)) - \bar\mu B^n(t)] + H(t)-\bar\mu B^n(t) \\
& =[\bar E^n(t) - H(t)] -[\bar S^n(B^n(t)) - \bar\mu B^n(t)] + [H(t)-\bar\mu t] + \bar\mu I^n(t),
\end{align*}
where $I^n(0) =0$, $I^n(t)$ is nondecreasing in $t$, and it increases only when $\bar Q^n(t) = 0.$ Thus $(\bar Q^n, \bar\mu I^n)$ is the unique solution to the one-dimensional Skorokhod problem associated with
%\[
%(\bar Q(t), \mu I^n(t)) = (\Phi, \Psi)(X^n)(t),\ t\in [0, T],
%\]
%where $(\Phi, \Psi)$ is the one-dimensional Skorokhod map, and
\[
X^n(t) = [\bar E^n(t) - H(t)] -[\bar S^n(B^n(t)) - \bar\mu B^n(t)] + [H(t)-\bar\mu t], t\ge 0.
\]
Thus $(\bar Q^n(t), \bar\mu I^n(t)) = (\Phi, \Psi)(X^n)(t), t\in[0, T]$, where $(\Phi, \Psi)$ is the one-dimensional Skorokhod map defined in \eqref{eqn:skorokhod_map}. 
{We also note that for $t\ge T$, letting $s=t-T,$
\begin{align*}
\bar Q^n(s+T)& = \bar E^n(T) -  \bar S^n(B^n(s+T)) \\
& = \bar Q^n(T) - [ \bar S^n(B^n(s+T)) - \bar S^n(B^n(T))] \\
& = \bar Q^n(T) - [ \bar S^n(B^n(s+T)) - \bar S^n(B^n(T)) - \bar\mu(B^n(s+T)-B^n(T))] - \bar\mu s \\
& \quad + \bar\mu(I^n(s+T) - I^n(T)).
\end{align*}
Define for $s\ge 0,$
\[
\tilde X^n(s) = \bar Q^n(T) - [ \bar S^n(B^n(s+T)) - \bar S^n(B^n(T)) - \bar\mu(B^n(s+T)-B^n(T))] - \bar\mu s,
\]
and
\[
\tilde Q^n(s) = \bar Q^n(s+T), \ \ \tilde I^n(s) =I^n(s+T) - I^n(T).
\]
Then $\tilde I^n(0) = 0$, and $\tilde I(s)$ is decreasing in $s$ and increases only when $\tilde Q^n(s) =0$, which yields that $
(\tilde Q^n, \bar\mu \tilde I^n)$ is the unique solution to the one-dimensional Skorokhod problem associated with $\tilde X^n$, and $(\tilde Q^n(t), \bar\mu \tilde I^n(t)) = (\Phi, \Psi)(\tilde X^n)(t), t\ge 0.$
}
At last, for $t<0$,
\begin{align*}
\bar Q^n(t) = \bar E^n(t).
\end{align*}
From (i) and using the strong law of large numbers (SLLN) for renewal processes $\bar S^n$, we have for $\tau > 0$,
\begin{align}
\sup_{t\in [0, \tau]} |X^n(t) - [H(t)-\bar\mu t]| & \to 0, \ \mbox{almost surely,} \\
\sup_{t\in [0, \tau]} |\tilde X^n(t) - \bar Q^n(T) + \bar\mu t| & \to 0, \ \mbox{almost surely.}
\end{align}
Using the Lipschitz continuity of the one dimensional Skorokhod map, we have the convergence of $(\bar Q^n, I^n)$ to $(q, i)$ as in (ii). We now show that $O^n(T) \to q(T)/\mu$ almost surely. We note that the clinic operation time is equal to
\[
T + O^n(T) = \sum_{i=1}^{E^n(T)} \nu^n_i + I^n(T+ O^n(T)) =  \sum_{i=1}^{E^n(T)} \nu^n_i + I^n(T).
\]
From the SLLN for i.i.d. random variables, we have
\[
\frac{1}{n}\sum_{i=1}^n (n\nu^n_i) \to \frac{1}{\mu}, \ \mbox{almost surely.}
\]
Using the continuous mapping theorem and the convergence for $\bar E^n$ in (i) and $I^n$ in (ii), we have
$T + O^n(T) \to \frac{H(T)}{\mu} + i(T),$
implying
\[
O^n(T)\to \frac{H(T)}{\mu} + i(T) - T = \frac{q(T)}{\mu}, \ \mbox{almost surely.}
\]
\end{proof}

\begin{proof}[Proof of Corollary \ref{opti_prop}.]
From Proposition \ref{prop:fluid-limits}, we have
\begin{align*}
&\frac{1}{n}\left(  r^n E^n(T) - c_w^n \int_0^\infty Q^n(t) dt - c_i^n I^n(T) - c_o^n O^n(T)\right)\\
& =   r^n \bar E^n(T) - c_w^n \int_0^{T+O^n(T)} \bar Q^n(t) dt - \frac{c_i^n}{n} I^n(T) - \frac{c_o^n}{n} O^n(T)\\
& \to \bar r H(T) - \bar c_w \int_0^{T+ q(T)/\bar\mu} q(t) dt - \bar c_i i(T) - \bar c_o \frac{q(t)}{\bar\mu}, \ \mbox{almost surely.}
\end{align*}
To show the corollary, we establish the uniform integrability of $\bar E^n(T), \int_0^{T+O^n(T)}\bar Q^n(t)dt, I^n(T)$ and $O^n(T).$ It suffices to show that their second moments are finite unformly for each $n$. We first observe that 
\begin{align*}
    \EE[(\bar E^n(T))^2]  \le (m^n/n)^2 < \infty,\ \mbox{and}  \  \EE[(I^n(T))^2]  \le T^2 < \infty.
\end{align*}
Next we note that
\begin{align*}
    \EE[(O^n(T))^2] & = \EE\left[\left(\sum_{i=1}^{E^n(T)}v^n_i + I^n(T) - T\right)^2\right] \le 3 \EE\left[\left(\sum_{i=1}^{E^n(T)}v^n_i\right)^2\right] + 3 \EE[(I^n(T))^2] + 3 T^2  \\
    & \le 3 \EE\left[\EE\left[\left(\sum_{i=1}^{E^n(T)}v^n_i\right)^2 \bigm|E^n(T)\right]\right] + 6 T^2 \le 3 \EE[(E^n(T)/\mu^n)^2 + E^n(T)(\sigma^n)^2] + 6 T^2 \\
    & \le 3 (m^n/\mu^n)^2 + 3 m^n (\sigma^n)^2 + 6 T^2. 
\end{align*}
From Assumption \ref{htc}, $\sup_n \EE[(O^n(T))^2] < \infty.$ Finally, we have 
\begin{align*}
    \EE\left[\left(\int_0^{T+O^n(T)}\bar Q^n(t)dt\right)^2\right] & \le \EE\left[\left(\int_0^{T+O^n(T)}\bar E^n(t)dt\right)^2\right] \le \EE\left[\left(\frac{m^n}{n}(T+O^n(T))\right)^2\right]\\
    & \le 2(m^n/n)^2 (T^2 + \EE[(O^n(T))^2]),
\end{align*}
which gives $\sup_n \EE[(\int_0^{T+O^n(T)}\bar Q^n(t)dt)^2]<\infty.$
This shows the result. 
\end{proof}

\begin{proof}[Proof of Lemma \ref{FCP:finiteoptimalvalue}.] We note that the for any feasible solution $\{A(t), t\in [0, T]\}$ to the FCP, the actual fluid arrival process $H(t) = \int_0^T F(t-s,s)d A(s)$ is a feasible solution to the control problem without punctuality defined in \eqref{eq:JH}. From Proposition \ref{fcp:no:pun}, the optimal value of the FCP $\sup_A J(A;\bar{\mathcal{M}})\le \check{J}^*(\bar{\mathcal{M}})<\infty.$
\end{proof}

\begin{proof}[Proof of Theorem \ref{th:optimal}.]
We first show \eqref{optimal_2}. As the first step, we show that the relative frequency process under $(m^{n,*}, \{a^{n,*}_i\}_{i=1}^{m^{n,*}})$ converges to $A^*(t)/A^*(T).$ Define the relative frequency process under $(m^{n,*}, \{a^{n,*}_i\}_{i=1}^{m^{n,*}})$ as
\begin{align}\label{asym-schedule-1}
\bar A^{n,*}(t) = \frac{1}{m^{n,*}}\sum_{i=1}^{m^{n,*}} 1_{\{a^{n,*}_i \le t\}} = \frac{k}{m^{n,*}} \ \mbox{when $k/m^{n,*}< A^*(t)/A^*(T) \le (k+1)/m^{n,*}.$}
\end{align}
Then
\begin{align}\label{asym-schedule-2}
\sup_{t\in[0,T]}|\bar A^{n,*}(t) - A^*(t)/A^*(T)| \le 1/m^{n,*} \to 0, \ \mbox{as $n\to\infty.$}
\end{align}
Furthermore, $m^{n,*}/n \to A^*(T)$. Thus from Corollary \ref{opti_prop}, we have
% \begin{align}\label{optimal_H}
% H(t) =  \int_0^T F(t-s, s) dA^*(t).
% \end{align}
% From Lemmas \ref{th:br}, and \ref{opti_prop}, we have
\begin{align}\label{optimal_obj}
\bar{\mathcal{J}}^n(m^{n,*}, \{a^{n,*}_i\}_{i=1}^{m^{n,*}}; \data^n) \to \bar r H(T) - \bar c_w \int_0^\infty q(t) dt - \bar c_i i(T) - \bar c_o \frac{q(T)}{\bar \mu},
\end{align}
where $(q(t), i(t))$ is the unique solution to the Skorokhod problem associated with $H(t)-\bar\mu t$, where $H(t) =  \int_0^T F(t-s, s) dA^*(t)$. Noting that $A^*(t)$ is an optimal solution of the FCP, we have
\[
\bar{\mathcal{J}}^n(m^{n,*}, \{a^{n,*}_i\}_{i=1}^{m^{n,*}}; \data^n) \to J^*(\bar\data).
\]
We now prove \eqref{optimal_1}. Recall from \eqref{H_n} that
\begin{align*}
H^n(t) = \frac{m^n}{n} \int_{0}^T F(t-s, s) d\bar A^n(s) = \frac{1}{n}\sum_{i=1}^{m^n} F(t- a_i^n, a^n_i), \ t\in \RR.
\end{align*}
From the proof of Proposition \ref{prop:fluid-limits}(i), we have
\begin{align}\label{lowerb-eq1}
\sup_{t\in \mathbb{R}} |\bar E^n(t) - H^n(t)| \to 0, \ \mbox{almost surely}.
\end{align}
Following the proof of Proposition \ref{prop:fluid-limits}(ii), with $H(t)$ replaced by $H^n(t)$, we have any $\tau >0,$
\begin{align}\label{lowerb-eq2}
\sup_{t\le \tau} |(\bar Q^n(t), I^n(t)) - (q^n(t), i^n(t))| \to 0, \ \mbox{almost surely,}
\end{align}
where $(q^n(t), i^n(t))$  is the unique solution to the Skorokhod problem associated with $\{H^n(t)-\bar\mu t, t\in [0, T]\}$.
Finally, following the proof of Corollary \ref{opti_prop}, we will show
\begin{align}\label{lowerb-eq3}
O^n(T) - q^n(T)/\bar\mu \to 0.
\end{align}
In fact, we have
\begin{align}\label{over:fluid}
\begin{aligned}
O^n(T) - q^n(T)/\bar\mu & =  \sum_{i=1}^{E^n(T)} \nu^n_i -T+ I^n(T) - \left(\frac{H^n(T)}{\bar\mu} + i^n(T) - T\right)\\
& =  \sum_{i=1}^{m^n} \nu^n_i - \frac{m^n/n}{\bar\mu}+ \frac{\bar E^n(T)-H^n(T)}{\bar\mu} + [I^n(T)- i^n(T)]\\
%& = \sum_{i=1}^{m^n} \nu^n_i - \frac{\bar m}{\bar\mu}+  [I^n(T)- i^n(T)]\\
& \to 0, \ \mbox{almost surely}.
\end{aligned}
\end{align}
Next noting that
\begin{align*}
\bar{\mathcal{J}}^n(m^n, \{a^n_i\}_{i=1}^{m^n}; \data^n)
& =   r^n \bar E^n(T) - c_w^n \int_0^{T+O^n(T)} \bar Q^n(t) dt - \frac{c_i^n}{n} I^n(T) - \frac{c_o^n}{n} O^n(T),
\end{align*}
and applying \eqref{lowerb-eq1}--\eqref{lowerb-eq3}, we have almost surely,
\begin{align*}
\left|\bar{\mathcal{J}}^n(m^n, \{a^n_i\}_{i=1}^{p^n}; \data^n) -\left[ \bar r H^n(T) - \bar c_w \int_0^{T+ q^n(T)/\bar \mu} q^n(t) dt - \bar c_i i^n(T) - \bar c_o \frac{q^n(t)}{\mu}\right]\right| \to 0.
\end{align*}
Finally, we observe that for each $n$,
\[
\bar r H^n(T) - \bar c_w \int_0^{T+ q^n(T)/\bar \mu} q^n(t) dt - \bar c_i i^n(T) - \bar c_o \frac{q^n(t)}{\bar \mu}=J(\bar A^n; \bar\data) \le J^*(\bar\data).
\]
Therefore,
\[
\limsup_{n\to\infty} \bar{\mathcal{J}}^n(p^n, \{a^n_i\}_{i=1}^{p^n}; \data^n) \le J^*(\bar\data).
\]

\end{proof}

% \begin{lemma}
% Assume that $\lim_{n\to\infty}n^{-1}\sum_{i=1}^{m^{n,*}}F(t-a^{n,*}_i, a^{n,*}_i)(1-F(t-a^{n,*}_i, a^{n,*}_i)) = \sigma^2(t) > 0$ for all $t\ge 0$. Then $\hat E^n$ converges weekly to a driftless Brownian motion with variance $\sigma^2(t)$ at time $t.$
% \end{lemma}
\begin{proof}[Proof of Proposition \ref{arrival:clt}.]
We first note that from the proof of Proposition \ref{prop:fluid-limits},
\begin{align}
& \sqrt{n}(\bar E^n(t) - H(t))\nonumber \\
& = \sqrt{n}\int_0^T F(t-s, s) d [\bar A^{n,*}(s)-A^*(s)/A^*(T)] + \frac{1}{\sqrt{n}} \sum_{i=1}^{m^{n,*}} [1_{\{U_i + a_{i}^{n,*} \le t\}} - F(t -a_{i}^{n,*}, a^{n,*}_i)], \label{clt-1}
\end{align}
where $\bar A^{n,*}$ is defined in \eqref{asym-schedule-1}. 
% Define $0 = t_0 < t_1 < \cdots < t_{m^{n,*}} \le T$ such that $A^*(t_k)/A^*(T) = k/m^{n,*}.$ 
From \eqref{asym-schedule-1} and \eqref{asym-schedule-2}, we have 
\begin{align*}
& \sup_{t\in\RR}\sqrt{n}\left|\int_0^T F(t-s, s) d [\bar A^{n,*}(s)-A^*(s)/A^*(T)] \right|\\
& \le \sqrt{n} \sum_{i=0}^{m^{n,*}-1}\sup_{t\in\RR}\left|\int_{a^{n,*}_i}^{a^{n,*}_{i+1}} F(t-s, s) d [A^*(s)/A^*(T)-\bar A^{n,*}(s)] \right|\\
& = \sum_{i=0}^{m^{n,*}-1}\sup_{t\in \RR} \left|\int_{a^{n,*}_i}^{a^{n,*}_{i+1}} F(t-s, s) d [\sqrt{n} (A^*(s)/A^*(T)-k/m^{n,*})] \right|.
\end{align*}
We note that $\sqrt{n} (A^*(s)/A^*(T)-k/m^{n,*}), s\in [a^{n,*}_i, a^{n,*}_{i-1}),$ is nondecreasing in $s$, and thus its total variation is $\sqrt{n} (A^*(s)/A^*(T)-k/m^{n,*}) \le \sqrt{n}/m^{n,*}.$ Thus following the above estimate,
\begin{equation}
 \sup_{t\in\RR}\sqrt{n}\left|\int_0^T F(t-s, s) d [\bar A^{n,*}(s)-A^*(s)/A^*(T)] \right| \le  \sum_{i=0}^{m^{n,*}-1} (a^{n,*}_{i+1}-a^{n,*}_i) \frac{\sqrt{n}}{m^{n,*}} 
 \le \frac{T\sqrt{n}}{m^{n,*}} \to 0. \label{new_est_1}
\end{equation} 
Define 
\begin{align*}
\hat Z^n(t) = \frac{1}{\sqrt{n}}\sum_{i=1}^{m^{n,*}} (1_{\{U_i + a^{n,*}_i \le t\}} - F(t- a^{n,*}_i, a^{n,*}_i)).
\end{align*}
It suffices to show that $\hat Z^n$ converges to a Gaussian process with mean $0$ and covariance $\sigma^2(s, t).$ We follow Theorem 15.4 in \cite{billingsley2013convergence} to establish the finite distribution convergence and the modulus of continuity condition.

We first show that for any $0 \le t_1 < \cdots < t_k \le T$, 
$(\hat Z^n(t_1), \ldots, \hat Z^n(t_k))$ converges weakly to a $k$-dimensional normal distribution with mean $0$ and covariance matrix $\Sigma(t_1, \ldots, t_k)$ with the $(i,j)$th entry being $\sigma^2(t_i, t_j)$. We note that for $s, t\in [0, T],$
\begin{align*}
\Cov(\hat Z^n(t), \hat Z^n(s)) & = \frac{1}{n}\sum_{i=1}^{m^{n,*}}\sum_{j=1}^{m^{n,*}} \Cov(1_{\{U_i + a^{n,*}_i \le t\}}, 1_{\{U_j + a^{n,*}_j \le s\}})\\
& = \frac{1}{n}\sum_{i=1}^{m^{n,*}} \Cov(1_{\{U_i + a^{n,*}_i \le t\}}, 1_{\{U_i + a^{n,*}_i \le s\}}) \\
& = \frac{1}{n}\sum_{i=1}^{m^{n,*}} [F(t\wedge s - a^{n,*}_i, a^{n,*}_i) - F(t - a^{n,*}_i, a^{n,*}_i)F(s - a^{n,*}_i, a^{n,*}_i)]\\
& \to \sigma^2(t, s).
\end{align*}
We next check the Lindeberg condition, and show that for any $\epsilon > 0$,
\begin{align*}
    \frac{1}{n} \sum_{k=1}^{m^{n,*}} \EE\left[\sum_{l=1}^k(\hat Z^n(t_l))^2 1_{\{\sum_{l=1}^k(\hat Z^n(t_l))^2 \ge \epsilon \sqrt{n}\}}\right] \to 0. 
\end{align*}
We note that by Holder's inequality,
\begin{align*}
   \EE\left[\sum_{l=1}^k(\hat Z^n(t_l))^2 1_{\{\sum_{l=1}^k(\hat Z^n(t_l))^2 \ge \epsilon \sqrt{n}\}}\right] \le \sqrt{\EE\left[\left(\sum_{l=1}^k(\hat Z^n(t_l))^2\right)^2 \right]\PP\left(\sum_{l=1}^k(\hat Z^n(t_l))^2 \ge \epsilon \sqrt{n}\right)}. 
\end{align*}
Let $\Delta^n_i(t_l) = 1_{\{U_i + a^{n,*}_i \le t_l\}} - F(t_l- a^{n,*}_i, a^{n,*}_i)$ which has mean $0$. We first note that 
\begin{align*}
    \EE\left[\left(\sum_{l=1}^k(\hat Z^n(t_l))^2\right)^2 \right] & \le k\sum_{l=1}^k \EE[(\hat Z^n(t_l))^4] \\
    & = k\sum_{l=1}^k \EE\left[\left(\frac{1}{\sqrt{n}}\sum_{i=1}^{m^{n,*}} (1_{\{U_i + a^{n,*}_i \le t_l\}} - F(t_l- a^{n,*}_i, a^{n,*}_i))\right)^4\right]\\
    & = k\sum_{l=1}^k \frac{1}{n^2} \sum_{i_1=1}^{m^{n,*}}\sum_{i_2=1}^{m^{n,*}}\sum_{i_3=1}^{m^{n,*}}\sum_{i_4=1}^{m^{n,*}} \EE[\Delta^n_{i_1}(t_l)\Delta^n_{i_2}(t_l)\Delta^n_{i_3}(t_l)\Delta^n_{i_4}(t_l)] \\
    & = k\sum_{l=1}^k \frac{1}{n^2} \sum_{i=1}^{m^{n,*}} \sum_{j=1}^{m^{n,*}}\EE[\Delta^n_i(t_l)^2\Delta^n_j(t_l)^2] \\
    & \le 16 k^2 (m^{n,*})^2/n^2 < \infty.
\end{align*}
We next observe that 
\begin{align*}
    \PP\left(\sum_{l=1}^k(\hat Z^n(t_l))^2 \ge \epsilon \sqrt{n}\right)&\le \frac{1}{\epsilon\sqrt{n}} \sum_{l=1}^k \EE[(\hat Z^n(t_l))^2] = \frac{1}{\epsilon\sqrt{n}} \sum_{l=1}^k \sum_{i=1}^{m^{n,*}}\frac{1}{n} \sum_{j=1}^{m^{n,*}}\EE[\Delta^n_i(t_l)\Delta^n_j(t_l)] \\
    & = \frac{1}{\epsilon n^{3/2}} \sum_{l=1}^k \sum_{i=1}^{m^{n,*}} \EE[\Delta^n_i(t_l)^2] \le \frac{1}{\epsilon n^{3/2}} \cdot 4 k m^{n,*}\to 0.
\end{align*}
Hence, from the Lindeberg central limit theorem, we have 
\begin{align}
    (\hat Z^n(t_1), \ldots, \hat Z^n(t_k)) \Go N(0, \Sigma(t_1, \ldots, t_k)),
\end{align}
or equivalently,
\begin{align}
    (\hat Z^n(t_1), \ldots, \hat Z^n(t_k)) \Go (\hat E(t_1), \ldots, \hat E(t_k)),
\end{align}
where $\hat E$ is a Gaussian process with mean $0$ and covariance $\sigma^2(s,t) = \Cov(\hat E(s), \hat E(t)).$

%In step 2, we show that $\EE[|\hat Z^n(t)-\hat Z^n(t_1)||\hat Z^n(t_2) - \hat Z^n(t)|] \le [G(t_2) - G(t_1)]^{2\alpha}$ for $t_1 < t < t_2$ and $n\ge 1$, where $\alpha > 1/2$ and $G$ is a nondecreasing continous function on $[0, 1]$. 

Next following Theorem 15.4 in \cite{billingsley2013convergence} we show that for each positive $\epsilon$ and $\eta$ there exists a $\delta$ such that 
\begin{align}\label{mod:con}
    \PP(w''(\hat Z^n, \delta) \ge \epsilon) \le \eta,
\end{align}
where for a function $x$ that is right continuous with left limts $w''(x,\delta) = \sup\min\{|x(t)-x(t_1)|, |x(t_2)-x(t)|\}$ with the supremum over $\{(t_1, t, t_2)\in [0, T]: t_1 \le t \le t_2, t_2 - t_1 \le \delta\}.$ We first derive that 
\begin{align*}
    & \EE[(\hat Z^n(t) - \hat Z^n(t_1))^4]  = \frac{1}{n^2}\EE\left[\sum_{i=1}^{m^{n,*}} (\Delta^n_i(t) - \Delta^n_i(t_1))\right]^4 \\
    &= \frac{1}{n^2} \sum_{i=1}^{m^{n,*}}\sum_{j=1}^{m^{n,*}}\EE[(\Delta^n_i(t) - \Delta^n_i(t_1))^2(\Delta^n_j(t) - \Delta^n_j(t_1))^2] \\
    & = \frac{1}{n^2} \sum_{i=1}^{m^{n,*}}\EE[(\Delta^n_i(t) - \Delta^n_i(t_1))^4] + \frac{1}{n^2} \sum_{i=1}^{m^{n,*}}\sum_{j=1, j\neq i}^{m^{n,*}}\EE[(\Delta^n_i(t) - \Delta^n_i(t_1))^2]\EE[(\Delta^n_j(t) - \Delta^n_j(t_1))^2].
\end{align*}
Let $F(t,s,a) = F(t-a, a) - F(s-a, a)$ for $s < t$ and $s, t, a\in [0, T].$  Then 
\begin{align*}
    \EE[(\Delta^n_i(t) - \Delta^n_i(t_1))^4] & = F(t, t_1, a^{n,*}_i) - (3F(t, t_1, a^{n,*}_i)^4 - 6F(t, t_1, a^{n,*}_i)^3 + 4 F(t, t_1, a^{n,*}_i)^2)\\
    &\le F(t, t_1, a^{n,*}_i),
\end{align*}
and 
\begin{align*}
    \EE[(\Delta^n_i(t) - \Delta^n_i(t_1))^2] = F(t, t_1, a^{n,*}_i)-F(t, t_1, a^{n,*}_i)^2 \le F(t, t_1, a^{n,*}_i).
\end{align*}
Thus, we have 
% \begin{align*}
%     \EE[\hat Z^n(t) - \hat Z^n(t_1)]^4 \le 
% \end{align*}
% Define $G^n(t) = \frac{1}{n} \sum_{i=1}^{m^{n,*}} F(t, a^{n,*}_i)$. Then 
% \[\EE[\hat Z^n(t) - \hat Z^n(t_1)]^2 \le G^n(t) - G^n(t_1)\]
% Thus 
\begin{align*}
    \EE[(\hat Z^n(t) - \hat Z^n(t_1))^4]  & \le \frac{1}{n^2} \sum_{i=1}^{m^{n,*}}F(t, t_1, a^{n,*}_i)+\frac{1}{n^2}\sum_{i=1}^{m^{n,*}}\sum_{j=1, j\neq i}^{m^{n,*}}F(t, t_1, a^{n,*}_i)F(t, t_1, a^{n,*}_j)\\
    & \le \frac{1}{n} \sum_{i=1}^{m^{n,*}}F(t, t_1, a^{n,*}_i) \left(\frac{1}{n} + \frac{1}{n}\sum_{j=1}^{m^{n,*}}F(t, t_1, a^{n,*}_j)\right).
\end{align*}
Then 
 \begin{align*}
      \EE[|\hat Z^n(t) - \hat Z^n(t_1)|^2|\hat Z^n(t_2) - \hat Z^n(t)|^2] & \le \sqrt{\EE[(\hat Z^n(t) - \hat Z^n(t_1))^4]\EE[(\hat Z^n(t_2) - \hat Z^n(t))^4]} \\
     & \le  \frac{1}{n} \sum_{i=1}^{m^{n,*}}F(t_2, t_1, a^{n,*}_i) \left(\frac{1}{n} + \frac{1}{n}\sum_{j=1}^{m^{n,*}}F(t_2, t_1, a^{n,*}_j)\right)\\
     & \le \frac{c_0 m^{n,*}}{n^2} (t_2 - t_1) + \frac{(c_0 m^{n,*})^2}{n^2} (t_2 - t_1)^2.
 \end{align*}
From the proof of Theorem 15.6 in \cite{billingsley2013convergence}, \eqref{mod:con} holds. Therefore, we conclude that $\hat Z$ converges weakly to a Gaussian process with mean $0$ and covariance function $\sigma^2(s, t).$
\end{proof}

\begin{proof}[Proof of Theorem \ref{ht_error}.]
 We consider the asymptotically optimal schedule $(m^{n,*}, \{a_i^n\}_{i=1}^{m^{n,*}})$ constructed in \eqref{policy} for the $n$-th system, and the optimal solution $A^*(t)$ for the FCP. 
We recall that for $t\ge 0$, 
\begin{align*}
\bar{\mathcal{J}}^n(m^{n,*}, \{a^{n,*}_i\}_{i=1}^{m^{n,*}}; \data^n) = \mathbb{E}\left[ r^n \bar E^n(T) - c^n_w \int_0^\infty \bar Q^n(t) dt - \frac{c^n_i}{n} I^n(T) - \frac{c^n_o}{n} O^n(T) \right],
\end{align*}
and 
\begin{align*}
J^*(\bar\data) = J(A^*; \bar\data) = \bar r H(T) - \bar c_w \int_0^\infty q(t) dt - \bar c_i i(T) - \bar c_o \frac{q(T)}{\bar \mu}.
\end{align*}
Define 
\begin{align*}
J^{n,*}(\bar\data)  = r^n H(T) - c^n_w \int_0^\infty q(t) dt - \frac{c_i^n}{n} i(T) - \frac{c_o^n}{n} \frac{q(T)}{\bar \mu}.
\end{align*}
To show the theorem, it suffices to show 
\begin{align}\label{error_2}
\sqrt{n}(J^{n,*}(\bar\data)-J^*(\bar\data))=O(1),
\end{align}
and 
\begin{align}\label{error_1}
\sqrt{n}(\bar{\mathcal{J}}^n(m^{n,*}, \{a^{n,*}_i\}_{i=1}^{m^{n,*}}; \data^n) -  J^{n,*}(\bar\data)) = O(1).
\end{align} 
We first note that \eqref{error_2} follows from the condition \eqref{cost:para:clt}, and that $\sup_{t\ge 0} [H(t) + q(t) + i(t)] + \int_0^\infty q(t) dt < \infty.$ In the following we prove \eqref{error_1}. We next observe that 
\begin{align}
 &\sqrt{n}(\bar{\mathcal{J}}^n(m^{n,*}, \{a^{n,*}_i\}_{i=1}^{m^{n,*}}; \data^n) -  J^{n,*}(\bar\data)) \nonumber\\
 & = r^n \EE[\hat E^n(T)] - c^n_w \int_0^\infty \EE[\hat Q^n(t)]dt - \frac{c^n_i}{n}\EE[\hat I^n(T)] - \frac{c^n_o}{n}\EE[\hat O^n(T)], \label{estimate_goal}
 %& = r^n \EE[\sqrt{n}(\bar E^n(T)-H(T))] - c^n_w \int_0^\infty \EE[\sqrt{n}(\bar Q(t) - q(t))] dt \\
 %&\quad - \frac{c^n_i}{n}\EE[\sqrt{n}(I^n(T) - i(T))] - \frac{c^n_o}{n}\EE[\sqrt{n}(O^n(T) - q(T)/\bar\mu)].
\end{align}
where $\hat E^n(t) = \sqrt{n}(\bar E^n(t)-H(t)), \hat Q^n(t) = \sqrt{n}(\bar Q^n(t)-q(t)), \hat I^n(t) = \sqrt{n}(\bar I^n(t)-i(t)), \hat O^n(t) = \sqrt{n}(\bar O^n(t)-q(t)/\bar\mu), t\ge 0.$ It suffices to show each item in \eqref{estimate_goal} is finite. 

From Proposition \ref{arrival:clt}, we have $\hat E^n(T)$ converges weakly to a Gaussian random variable with mean $0$. In the following We show that $\{\hat E^n(T)\}_{n=1}^\infty$ is uniformly integrable (UI), that is for every $\epsilon > 0$ there exists $K>0$ such that 
$\EE[|\hat E^n(T)| 1_{\{|\hat E^n(T)| \ge K\}} ] \le \epsilon$ for all $n$. 
From \eqref{clt-1} and \eqref{new_est_1}, we have
\begin{align*}
\EE\left[|\hat E^n(T)| 1_{\{|\hat E^n(T)| \ge K\}} \right] &\le \EE\left[|\sqrt{n}T/m^{n,*}+\hat Z^n(T)| 1_{\{|\sqrt{n}T/m^{n,*}+\hat Z^n(T)| \ge K\}} \right] \\
& \le \sqrt{n}T/m^{n,*} + \EE\left[|\hat Z^n(T)| 1_{\{|\hat Z^n(T) \ge K- K_1\}} \right],
\end{align*}
where $K_1 = \sup_{n}\sqrt{n}T/m^{n,*}$ and $K$ is chosen to be larger than $K_1.$
Next using Holder's inequality, we have 
\begin{align*}
\sup_n\EE\left[|\hat Z^n(T)| 1_{\{|\hat Z^n(T) \ge K- K_1\}} \right] & \le \sup_n \sqrt{\EE\left[(\hat Z^n(T))^2\right]\PP(|\hat Z^n(T)| \ge K-K_1)} \\
& \le \sup_n \frac{1}{K-K_1} \EE\left[(\hat Z^n(T))^2\right] \le \sup_n \frac{4 m^{n,*}}{n(K-K_1)} \to 0, \ \mbox{as $K\to\infty.$}
\end{align*}
Thus, we have $\EE[|\hat E^n(T)|] \to 0.$ 

Next let $\mathcal{G}_k = \sigma(U_1, \ldots, U_k)$. Then $\hat Z^n(t)$ is a martingale with respect to $\mathcal{F}^n_t \equiv \mathcal{G}_{E^n(t)+1}$. Using Doob's inequality, we have 
\begin{align}\label{arrival_estimate_second_order}
\EE\left[\left(\sup_{t\in[0,T]}|\hat Z^n(t)|\right)^2\right] \le 4 \EE[(\hat Z^n(T))^2] \le 16 m^{n,*}/n < \infty. 
\end{align}
We note that for $t\in[0,T],$
\begin{align*}
\hat Q^n(t) & = \sqrt{n}(\bar E^n(t) - H(t)) - \sqrt{n}(\bar S^n(B^n(t)-\bar \mu B^n(t)) + \sqrt{n}(H(t)-t) + \sqrt{n}\bar\mu I^n(t)   \\
& \quad -\sqrt{n} (H(t) - t + \bar\mu i(t)) \\
& = \Phi(\hat Y^n +\sqrt{n} y)(t) - \Phi(\sqrt{n}y)(t),
\end{align*}
where $\hat Y^n(t) = \hat E^n(t) - \sqrt{n}(\bar S^n(B^n(t)-\bar \mu B^n(t))$, and $y(t) = H(t) -t.$ 
Similarly, we can show that 
\[\hat I^n(t) = \Psi(\hat Y^n + \sqrt{n}y)(t) - \Psi(\sqrt{n}y)(t), t\in [0, T].\]
From the Lipchitz continuity of the one-dimensional Skorohkod map, we have for some $C_1 > 0$, 
\begin{align}
    & \sup_{t\in[0, T]} |\hat Q^n(t)| \le C_1 \sup_{t\in[0, T]} |\hat Y^n(t)|\le C_1 \sup_{t\in[0,T]}|\hat Z^n(t)| + C_1 \sup_{t\in [0,T]}|\sqrt{n}(\bar S^n(B^n(t)-\bar \mu B^n(t))|, \label{queue_estimate_first_order}\\
    & \sup_{t\in[0, T]} |\hat I^n(t)| \le C_1 \sup_{t\in[0, T]} |\hat Y^n(t)|\le C_1 \sup_{t\in[0,T]}|\hat Z^n(t)| + C_1 \sup_{t\in [0,T]}|\sqrt{n}(\bar S^n(B^n(t)-\bar \mu B^n(t))|. \label{idle_estimate_first_order}
\end{align}
From Theorem 4 in \cite{krichagina92}, for the renewal process $S^n(t)$, we have for $C_2 > 0$, 
\begin{align}\label{renewal_estimate_second_order}
\EE\left[\left(\sup_{t\in [0,T]}|\sqrt{n}(\bar S^n(B^n(t)-\bar \mu B^n(t))|\right)^2\right]\le \EE\left[\sup_{t\in [0, T]} |\sqrt{n}(\bar S^n(t) - \bar \mu t)|^2\right] \le C_2 (1+ T).
\end{align}
Applying \eqref{renewal_estimate_second_order} and \eqref{arrival_estimate_second_order} in \eqref{queue_estimate_first_order} and \eqref{idle_estimate_first_order}, we have 
\begin{align}\label{queue_idle_second_order}
  \EE\left( \sup_{t\in[0, T]} |\hat Q^n(t)|^2\right) + \EE\left( \sup_{t\in[0, T]} |\hat I^n(t)|^2\right) < \infty.  
\end{align}
%
% From Proposition \ref{cost:para:clt}, $\hat E^n(t) = \sqrt{n}(\bar E(t) - H(t))$ converges weakly to $\hat E(t)$, where $\hat E$ is a Gaussian process with mean $0$ and covariance function $\sigma^2(s, t).$ Now the scaled centered renewal process $\sqrt{n}(\bar S^n(t) - \bar\mu t)$ converges weakly to $\bar\sigma \bar\mu^3 W(t)$, where $W$ is a standard Brownian motion independent of $\hat E.$ Thus, we have 
% \begin{align}
%     \hat Y(t) \to \hat E(t) - \bar\sigma \bar\mu^3 W(t-i(t)).
% \end{align}
% From Theorem 2 and Lemma 2 in \cite{honnappa2015}, we have $\sqrt{n}(I^n(t)-t)$ converges weakly to $\hat I(t)$, where $\hat I$ can be determined by the directional derivative of the Skorohkod map $(\Phi, \Psi)$ defined in (19) of \cite{honnappa2015}. Then $(\sqrt{n}(\bar Q^n(t) - q(t)), \sqrt{n}(I^n(t)-t))$ converges weakly to $(\hat Y(t) + \hat I(t), \hat I(t)).$
%
We next note that similar to \eqref{over:fluid}, 
\begin{align*}
\hat O^n(T) & =  \sqrt{n}(O^n(T) - q(T)/\bar\mu)  = \sqrt{n}\left( \sum_{i=1}^{E^n(T)} \nu^n_i - T + I^n(T) - \frac{1}{\bar \mu}(H(T)-\bar\mu T + \bar\mu i(T)) \right)\\
& = \sqrt{n}\left( \sum_{i=1}^{E^n(T)} \nu^n_i - \frac{E^n(T)}{\mu^n} \right) + \sqrt{n}\left(\frac{E^n(T)}{\mu^n} - \frac{H(T)}{\bar\mu} \right) + \sqrt{n}(I^n(T) - i(T))\\
& = \sqrt{n}\left( \sum_{i=1}^{E^n(T)} \nu^n_i - \frac{E^n(T)}{\mu^n} \right) + \frac{\hat E^n(T)}{\bar\mu^n}  - \frac{H(T)\hat\mu^n}{\bar\mu^n\bar\mu}  + \hat I^n(T).
\end{align*}
% Using the functional CLT for i.i.d. sum and continuous mapping theorem, we have 
% \[\sqrt{n}\left( \sum_{i=1}^{E^n(T)} \nu^n_i - \frac{\bar E^n(T)}{\bar\mu} \right) \Go \bar\sigma \tilde W(\bar m T).\]
%
Using conditional expectation, we observe that 
\begin{align*}
\sup_n \EE\left[\sqrt{n}\left( \sum_{i=1}^{E^n(T)} \nu^n_i - \frac{E^n(T)}{\mu^n} \right) \right]^2 & = \sup_n \EE\left[\EE\left[\sqrt{n}\left( \sum_{i=1}^{E^n(T)} \nu^n_i - \frac{E^n(T)}{\mu^n} \right) \bigm| E^n(T)\right]^2\right]\\
& = \sup_n \EE[\bar E^n(T) (\bar\sigma^n)^2] < \infty.
\end{align*}
Hence we have 
\begin{align}\label{overtime_estimate_second_order}
\sup_n \EE[|\hat O^n(T)|^2] < \infty.
\end{align}
In view of \eqref{estimate_goal}, it is only left to show 
\[
\int_T^\infty \EE[\hat Q^n(t)] dt < \infty.
\]
We note that for $t\ge T$, 
\begin{align*}
    \hat Q^n(t) & = \hat Q^n(t) 1_{\{T\le t \le T + O^n(T)\wedge q(T)/\bar\mu\}} \\
    & \quad + \hat Q^n(t) 1_{\{ T+ O^n(T)\wedge q(T)/\bar\mu < t \le T+ O^n(T)\vee q(T)/\bar\mu\}} + \hat Q^n(t) 1_{\{ t > T+ O^n(T)\vee q(T)/\bar\mu\}} \\
    & = \left[\hat Q^n(T) - \sqrt{n}(\bar S^n(B^n(t))-\bar S^n(B^n(T)))\right] 1_{\{T\le t \le T+ O^n(T)\wedge q(T)/\bar\mu\}}\\
    & \quad + \hat Q^n(t) 1_{\{ T+ O^n(T)\wedge q(T)/\bar\mu < t \le T+ O^n(T)\vee q(T)/\bar\mu\}},
\end{align*}
and 
\begin{align*}
    \EE\left[\int_T^\infty \hat Q^n(t) dt \right] & =  \EE\left[\int_T^\infty \hat Q^n(t) 1_{\{T\le t \le T+ O^n(T)\wedge q(T)/\bar\mu\}} dt \right] \\
    & \quad + \EE\left[\int_T^\infty \hat Q^n(t) 1_{\{ T+ O^n(T)\wedge q(T)/\bar\mu < t \le T+ O^n(T)\vee q(T)/\bar\mu\}} dt \right]\\
    & \le \EE\left[\int_T^{T+ q(T)/\bar\mu}|\hat Q^n(T)| +| \sqrt{n}(\bar S^n(B^n(t))-\bar S^n(B^n(T))) | dt \right] \\
    & \quad + \EE\left[\sqrt{n}\bar Q^n(T) |O^n(t) - q(T)/\bar\mu|\right].
\end{align*}
Using \eqref{renewal_estimate_second_order}, \eqref{queue_idle_second_order}, and \eqref{overtime_estimate_second_order}, we have 
\begin{align*}
    \EE\left[\int_T^\infty \hat Q^n(t) dt \right] \le q(T)/\bar\mu [\EE(|\hat Q^n(T)|) + C_2 (1+ t-T)] + \sqrt{\EE((\bar Q^n(T))^2)\EE((\hat O^n(T))^2)} < \infty.
\end{align*}
This completes the proof. 
\end{proof}

%\begin{lemma}\label{diff-refine}
%Under the conditions in \eqref{ht_1}, \eqref{ht_2}, \eqref{cost_par_scaling} and \eqref{app_cond}, as $n\to\infty$,
%\begin{align}
%(\hat E^n, \hat Q^n, \hat I^n) \Rightarrow
%\end{align}
%\end{lemma}

% \proof{Proof of Corollary \ref{fcp:uni:pun}.}
% This result easily follows from the proof of Proposition \ref{fcp:no:pun} that if $H(t) \geq \mu t$ with $r \leq \frac{c_o}{\mu}$, then the objective function in \eqref{eq:obj1} is less than or equal to $0$ and it is equal to $0$ if and only if $H(t) = \mu t$. Similarly, for $H(t) < \mu t$, nonzero idle cost will be accrued leading to a negative objective value.
% \end{proof}

\begin{proof}[Proof of Lemma \ref{lm:lindley_skorokhod}.]
  For $k=0$, $\hat{i}(0) = \frac{1}{\mu} \max \{ -H(0), 0 \} = 0$ and thus $\hat{q}(0) = \hat{H}(0)$. Thus, $q_0 = \hat{q}(0) = \hat{H}(0)$ and $I_0 = \hat{i}(0) = 0$. Now, assuming $q_k = \hat{q}(k)$ and $\displaystyle \sum_{j=0}^{k} I_j = \hat{i}(k)$ for $k = 0, 1, \dots, n-1 < K$, we must now show that this holds for $n$.

  To show that $q_n = \hat{q}(n)$, we note that by the induction hypothesis,
  \begin{align*}
    q_n & = \max \{ 0, \hat{q}(n-1) + a_n - \mu \frac{T}{K} \} \\
     & = \max \{0, \hat{H}(t_{n-1}) - \mu \frac{(n-1)T}{K} + \mu \hat{i}(n-1) + \hat{H}(t_{n}) - \hat{H}(t_{n-1}) - \mu \frac{T}{K} \} \\
     & = \max \{0, \hat{H}(t_{n}) - \mu \frac{nT}{K} + \mu \hat{i}(n-1)  \} \\
     & = \max \left\{ 0, \hat{H}(t_{n}) - \mu \frac{nT}{K} + \displaystyle \max_{i \in \{0, \dots, n\}} \max \left\{ 0, \mu \frac{iT}{K} - \hat{H}(t_{i}) \right\}  \right\} \\
     & = \hat{H}(t_{n}) - \mu \frac{nT}{K} + \mu \hat{i}(n) \\
     & = \hat{q}(n).
  \end{align*}
To show that $\sum_{j=0}^{n} I_j = \hat{i}(n)$, we note that by the induction hypothesis,
  \begin{align*}
    \sum_{j=0}^{n} I_j & = \hat{i}(n-1) + \frac{1}{\mu} \max \{ 0, \mu \frac{T}{K} - (\hat{H}(t_{n}) + \hat{H}(t_{n-1})) - q_n \} \\
     & = \frac{1}{\mu} \left( \displaystyle \max_{i \in \{0, \dots, n-1 \}} \max \left\{ 0, \mu \frac{iT}{K} - \hat{H}(t_{i}) \right\} + \max \left\{0, \mu \frac{T}{K} - (\hat{H}(t_{n}) + \hat{H}(t_{n-1})) - q_n  \right\}  \right) \\
     & = \frac{1}{\mu} \displaystyle \max_{i \in \{0, \dots, n \}} \max \left\{0, \mu \frac{iT}{K} - \hat{H}(t_{i})  \right\} \\
     & = \hat{i}(n).
  \end{align*}
\end{proof}

 \begin{proof}[Proof of Proposition \ref{th:qp_error_bound}.]
 Let $A(t)$ be an feasible control for the FCP, and $\mathbf{a}$ be the discretized control for the QP across a grid of length $K$ as such:
  $$
  \mathbf{a} = (A(t_1) - A(t_0), A(t_2) - A(t_1), \dots, A(t_K) - A(t_{K-1}) )^T,
  $$
  where $t_k ={kT}/{K}$, $k=0, 1, \dots K$. We first show that 
  \begin{align}\label{qp:con:1}
|J(A; \mathcal{M}) - \hat J_K(\mathbf{a};\hat{\mathcal{M}})| \to 0
  \end{align}
  
% We first show that for each $t\in [0, T]$, $\sup_{t\in [0, T]}|H(t) - \hat H(t)|\to 0$ as $K\to\infty$ and consequently, $J(A; \mathcal{M}) - \hat J_K(\mathbf{a}; \hat{\mathcal{M}})\to 0$ as $K\to\infty.$
%
 % We first note that 
 % \begin{align*}
 %     H(t) - \hat H(t) = \int_0^T F(t-s, s) dA(s) - \sum_{k=0}^{K-1} a_k F(t - t_k, t_k). 
 % \end{align*}
 % From Assumption \ref{unp_dist} and noting that $F(t,a)$ is continuous in $t$, for any $\epsilon > 0$, there exists $\delta > 0$ such that when $a, b$ are in the same partition and $|a-b| < \delta$, we have $\sup_{t\in [0, T]}|F(t-a, a) - F(t-b, b)| \le \sup_{t\in [0, T]}|F(t-a, a) - F(t-a, b)|+ \sup_{t\in [0, T]}|F(t-a, b) - F(t-b, b)|$
 
We note that since patients are not accepted into the system if arriving after time $T$, and thus have
  $$
  \int_{0}^{\infty}q(t)dt = \int_{0}^{T}q(t)dt + \frac{c_w}{2\mu} q(T)^2.
  $$
  Using the triangle inequality, we have
  \begin{align*}
    |J(A; \mathcal{M}) - \hat{J}_K(\mathbf{a}; \hat{\mathcal{M}})| & \leq  r |H(T) - \hat{H}(T)| + c_w \left| \frac{T}{K} \sum_{i=0}^{K-1} q_i - \int_{0}^{T} q(t) dt \right|  \\
     &   + \frac{c_w}{2 \mu} \left| q_K^2 - q(T)^2 \right| + c_I \left| i(T) - \sum_{i=0}^{K-1} I_i \right|  + \frac{c_o}{\mu} \left| q_k - q(T) \right|.
  \end{align*}
  We will examine the individual terms:
  \begin{enumerate}
    \item $\left| H(T) - \hat{H}(T) \right|$,
    \item $\left| i(T) - \sum_{i=0}^{K-1} I_i \right|$,
    \item $\left| q_K - q(T) \right|$,
    \item $\left| \frac{T}{K} \sum_{i=0}^{K-1}q_i - \int_{0}^{T} q(t) dt \right|$, and
    \item $\left| q_K^2 - q(T)^2 \right|$.
  \end{enumerate}
It follows naturally from the assumption in the proposition that the first term converges to $0$. 
  %\ref{assum:qp:convergence}, we have 
  % \[
  % \left| H(T) - \hat{H}(T) \right| \to 0, \mbox{as $K\to\infty$.}
  % \]
%
%   \textcolor{red}{
%   we refer to Theorem \ref{th:dragomirthm5}. We first note that $F$ is a 1-$\bar{f}$-H{\"o}lder type mapping where $\bar{f}$ is the maximum value the derivative of $F$ achieves on $[0,T]$. Further, $A(t)$ is a nondecreasing function on $[0, T]$ and is thus of bounded variation with $V_a^b(a) = A(T) < \infty$ which is guaranteed if 0 is within the convex hull of the support of $F$. Thus,
%   $$
%   \left| H(T) - \hat{H}(T) \right| \leq \frac{T}{K} r \bar{f} A(T).
%   $$
%   For the first term, defining a piecewise nondecreasing function $A_K(t) = A(t_k), t\in [t_{k-1}, t_k)$, then 
% \[
% \hat H(T) = \int_0^T F(t-s, s) d A_K(t), 
% \]
% and 
% \begin{align*}
%     |H(T) - \hat H(T)| & = \left|\int_0^T F(t-s, s) d [A(t) - A_K(t)] \right|\\
%     & \le \left|\sum_{k=0}^{K-1}\int_{t_k}^{t_{k+1}} F(t-s, s) d [A(t) - A_K(t)] \right|
% \end{align*}
% We observe that for $t\in [t_{k-1}, t_k)$, $A_(t)-A_K(t)$ is nondecreasing.
% }
For the second term, by Lemma \ref{lm:lindley_skorokhod}, we have
  \begin{align*}
    \left| i(T) - \sum_{i=0}^{K-1} I_i \right| & =   \left| i(T) - \hat{i}(K) \right| \\
     & =  \mu^{-1} \left| \displaystyle \sup_{0 \leq s \leq T} \max \left\{ \mu s - H(s), 0 \right\} - \displaystyle \max_{i \in \{0, 1, \dots, K \}} \max \left\{ \mu t_i - \hat{H} (t_i), 0 \right\} \right| \\
     & \leq \mu^{-1} \kappa_0 \left(\sup_{t\in[0,T]}|H(t)-\hat H(t)| + \frac{T}{K}\right) \to 0, %+ \frac{1}{\mu K}%\frac{T}{K} \left( 1 + \frac{\bar{f} A(T)}{\mu} \right).
  \end{align*}
  where $\kappa_0$ is the Lipschitz continuity coefficient for the one-dimensional Skorokhod map in Theorem \ref{sm:repres}. 
For the third term, we have
  \begin{align*}
    \left|q_K - q(T) \right| & =  \left| \hat{H}(T) + \mu \hat{i}(K) - \left( H(T) + \mu i(T)  \right) \right| \\
     & \leq  \left| \hat{H}(T) - H(T) \right| + \mu \left| \hat{i}(K) - i(T) \right| \to 0. 
    % & \leq  \frac{T}{K} \left( \bar{f} A(T) +  1 + \frac{\bar{f} A(T)}{\mu} \right) \\
    % & =  \frac{T}{K} \left( 1 + \frac{(\mu + 1) \bar{f} A(T)}{\mu} \right)
  \end{align*}
%  by using our results for the first 2 terms.
For the fourth term, we have
  \begin{align*}
    \left| \frac{T}{K} \sum_{k=0}^{K-1} q_i - \int_{0}^{T} q(t) dt \right| & =  \left| \frac{T}{K} \sum_{k=0}^{K-1} \left( \hat{H}(t_k) - \mu t_k + \mu \hat{i}(k) \right) - \int_{0}^{T} \left( H(t) - \mu t + \mu i(t) \right) dt \right| \\
     & \leq  \left| \frac{T}{K} \sum_{k=0}^{K-1} \hat{H}(t_k) - \int_{0}^{T} H(t) dt \right| + \left| \frac{T}{K} \sum_{k=0}^{K-1} \mu t_k - \int_{0}^{T} \mu t dt  \right| \\
     & + \left| \frac{T}{K} \sum_{k=0}^{K-1} \mu \hat{i}(k) - \int_{0}^{T} \mu i(t) dt \right|.
  \end{align*}
  We will examine the above three components below and using the approximation error of a left Riemann-sum, we have:
  \begin{align*}
    \left| \frac{T}{K} \sum_{k=0}^{K-1} \hat{H}(t_k) - \int_{0}^{T} H(t) dt \right| & =   \left| \frac{T}{K} \sum_{k=0}^{K-1} \hat{H}(t_k) - \int_{0}^{T} \hat{H}(t) dt + \int_{0}^{T} \hat{H}(t) dt - \int_{0}^{T} H(t) dt  \right| \\
     & \leq  \left| \frac{T}{K} \sum_{k=0}^{K-1} \hat{H}(t_k) - \int_{0}^{T} \hat{H}(t) dt \right| + \left|  \int_{0}^{T} \hat{H}(t) dt - \int_{0}^{T} H(t) dt \right| \\
     & \le  \sum_{k=0}^{K-1} \int_{t_k}^{t_{k+1}} [\hat{H}(t) - \hat{H}(t_k)] dt  + T \sup_{t\in[0, T]}|H(t) - \hat H(t)|\\
     & = \int_{t_k}^{t_{k+1}} [\hat{H}(t_K) - \hat{H}(t_0)] dt+ T \sup_{t\in[0, T]}|H(t) - \hat H(t)|\\
     & = \frac{T}{K}[\hat H(T) - \hat H(0)]+ T \sup_{t\in[0, T]}|H(t) - \hat H(t)| \to 0,
     %& \leq  \bar{h} \frac{T^2}{K} + \bar{f} \frac{T^2}{K} A(T) \\
     %& =  \frac{T}{K} \left( \bar{h}T + \bar{f} T A(T) \right),
  \end{align*}
and 
  $$
  \left| \frac{T}{K} \sum_{k=0}^{K-1} \mu t_k - \int_{0}^{T} \mu t dt \right| = \frac{T}{K} \mu T \to 0,
  $$
  and
  $$
  \left| \frac{T}{K} \sum_{k=0}^{K-1} \mu \hat{i}(k) - \int_{0}^{T} \mu i(t) dt \right| \leq  \frac{T^2}{K} \to 0.
  $$
  %$$
%  \begin{array}{rrl}
%    \left| \frac{T}{K} \sum_{k=0}^{K-1} \mu \hat{i}(k) - \int_{0}^{T} \mu i(t) dt \right| &  =  & \left| \frac{T}{K} \sum_{k=0}^{K-1} \mu \hat{i}(k) - \int_{0}^{T} \mu \hat{i}(k) dt + \int_{0}^{T} \mu \hat{i}(k) dt - \int_{0}^{T} \mu i(t) dt \right| \\
%     &  & \\
%     & \leq & \left| \frac{T}{K} \sum_{k=0}^{K-1} \mu \hat{i}(k) - \int_{0}^{T} \mu \hat{i}(k) dt \right| + \left| \int_{0}^{T} \mu \hat{i}(k) dt - \int_{0}^{T} \mu i(t) dt \right| \\
%     && \\
%     & \leq & \frac{T}{K} \mu T.
%  \end{array}
%  $$
  For the final and fifth term, noting that $q_K + q(T) < \infty$, we have
  \begin{align*}
    \left| q_K^2 - q(T)^2 \right| & =   \left| (q_K - q(T)) (q_K + q(T)) \right| \to 0.
     % & \leq \left| q_K + q(T) \right| \left( \left| \hat{H}(T) - H(T) \right| + \mu \left| \hat{i}(K) - i(T) \right|  \right) \\
     % & \leq \left| q_K + q(T) \right| \frac{T}{K} \left( \bar{f} A(T) + 1 + \frac{(\mu + 1) \bar{f} A(T) }{\mu} \right)\\
     % & \leq  \frac{T}{K} \left( 1 + \frac{(2 \mu + 1) \bar{f} A(T) }{\mu} \right) \left( 2 A(T) + 1 + \frac{(\mu + 1) \bar{f} A(T)}{\mu} \right).
  \end{align*}
Combining the five terms gives us \eqref{qp:con:1}.

Now let $\mathbf{a}^* = (a_0^*, \ldots, a_{K-1}^*)'$ be an optimal solution of the QP associated with the resolution of discretization $K$. Define 
\[
A^*_K(t)= \begin{cases}
    0, & t\in [0, t_1), \\
\sum_{i=0}^{k-1} a_i^*, & t\in [t_{k}, t_{k+1}), k = 1, \ldots, K, \\
A^*(T-), & t= T.
\end{cases} 
\]
We first note that $A^*_K(t)$ is a feasible control for the FCP, and from \eqref{qp:con:1}, $|J(A^*_K; \mathcal{M})-\hat J_K(\mathbf{a}^*; \hat{\mathcal{M}})|\to 0$ as $K\to\infty$. For any feasible solution $A(t)$ of the FCP and its corresponding discretization solution $\mathbf{a}$ of the QP, we have 
\begin{align*}
J(A; \mathcal{M}) & = \lim_{K\to \infty} \hat J_K(\mathbf{a};\hat{\mathcal{M}}) \\
& \le \liminf_{K\to\infty}  \hat J_K(\mathbf{a}^*;\hat{\mathcal{M}}) \\
& = \lim_{K\to\infty}[\hat J_K(\mathbf{a}^*; \hat{\mathcal{M}})-J(A^*_K; \mathcal{M})] + \liminf_{K\to\infty}J(A^*_K; \mathcal{M})\\
& = \liminf_{K\to\infty}J(A^*_K; \mathcal{M})
\end{align*}
We thus have that 
\begin{align}
J^*(\mathcal{M}) \le \liminf_{K\to\infty}J(A^*_K; \mathcal{M})=\liminf_{K\to\infty}  \hat J_K(\mathbf{a}^*;\hat{\mathcal{M}}).
\end{align}
On the other hand, for each $K$, we have $J(A^*_K; \mathcal{M}) \le J^*(\mathcal{M}).$ We conclude that $\limsup_{K\to\infty}J(A^*_K; \mathcal{M}) \le J^*(\mathcal{M})$, and
\begin{align*}
J^*(\mathcal{M}) = \lim_{K\to\infty}J(A^*_K; \mathcal{M}) =\lim_{K\to\infty}  \hat J_K(\mathbf{a}^*;\hat{\mathcal{M}}) = \lim_{K\to\infty}\hat J^*_K.
\end{align*}

\end{proof}

\end{document}